\documentclass[10pt,a4paper]{article}

\usepackage{amsmath,amssymb}
\usepackage{theorem,ifthen}
\usepackage{graphicx}
\usepackage{color}

\newtheorem{prop}{Proposition}[section]
\newtheorem{lemma}[prop]{Lemma}
\newtheorem{theo}[prop]{Theorem}
\newtheorem{coroll}[prop]{Corollary}
{\theorembodyfont{\rmfamily}
 \newtheorem{remark}[prop]{Remark}
 \newtheorem{definition}[prop]{Definition}
 \newtheorem{example}[prop]{Example}
}
\newcommand{\qedup}{\par\vspace{-1.5em}}

\newenvironment{proof}[1][Proof]{%
  \begin{list}{}{%
      \settowidth{\labelwidth}{\textit{#1.}}%
      \setlength{\itemindent}{\labelwidth}%
      \addtolength{\itemindent}{\labelsep}%
      \setlength{\leftmargin}{0pt}
      \setlength{\parsep}{0pt}
      \setlength{\listparindent}{\parindent}
    }\item[\textit{#1.}]}
  {\hspace*{\fill}$\Box$ \end{list}}
\newenvironment{proof*}[1][Proof]{%
  \begin{list}{}{%
      \settowidth{\labelwidth}{\textit{#1.}}%
      \setlength{\itemindent}{\labelwidth}%
      \addtolength{\itemindent}{\labelsep}%
      \setlength{\leftmargin}{0pt}
    }\item[\textit{#1.}]}
  {\end{list}}
\newenvironment{definition*}{\begin{definition}}
  {\hspace*{\fill}$\lrcorner$ \end{definition}}
\newenvironment{remark*}{\begin{emark}}
  {\hspace*{\fill}$\lrcorner$ \end{remark}}
\newenvironment{example*}{\begin{example}}
  {\hspace*{\fill}$\lrcorner$ \end{example}}

\newcommand{\braket}[2]{[#1|#2]}
\newcommand{\Braket}[2]{\Bigl[#1\Big|#2\Bigr]}
\newcommand{\Jorth}{{\langle\perp\rangle}}

\newcommand{\pmat}[1]{\begin{pmatrix}#1\end{pmatrix}}
\newcommand{\set}[2]{\{#1\,|\,#2\}}
\newcommand{\bigset}[2]{\bigl\{#1\,\big|\,#2\bigr\}}

\DeclareMathOperator{\mspan}{span}
\DeclareMathOperator{\Real}{Re}
\DeclareMathOperator{\Imag}{Im}

\newcommand{\proj}{\mathrm{pr}}

\newcommand{\N}{\mathbb{N}}

\newcommand{\R}{\mathbb{R}}
\newcommand{\C}{\mathbb{C}}

\newcommand{\mdef}{\mathcal{D}}
\newcommand{\range}{\mathcal{R}}
\newcommand{\eps}{\varepsilon}

\newcommand{\graph}{G}
\newcommand{\invgraph}{G_{\mathrm{inv}}}
\newcommand{\ri}{\mathrm{i}}

\newcommand{\cN}{\mathcal{N}}

\newcommand\wt{\widetilde} 

\begin{document}

\title{Dichotomous Hamiltonians with Unbounded Entries and Solutions
  of Riccati Equations
}
\author{Christiane Tretter\footnote{
    Institute of Mathematics,
    University of Bern, 
    Sidlerstrasse 5, 
    CH-3012 Bern, Switzerland, 
    \texttt{tretter@math.unibe.ch}} ,
  Christian Wyss\footnote{
    Department of
    Mathematics and Informatics,
    University of Wuppertal,
    Gau\ss stra\ss e 20,
    D-42097 Wuppertal,
    Germany,
    \texttt{wyss@math.uni-wuppertal.de}}}
\date{\today}
\maketitle

\thanks{\centerline{\emph{Dedicated to Rien Kaashoek on the occasion of his 75th birthday}}}

\begin{center}
  \parbox{11cm}{\small \textbf{Abstract.}  An operator Riccati
    equation from systems theory is considered in the case that all
    entries of the associated Hamiltonian are unbounded.  Using a
    certain dichotomy property of the Hamiltonian and its symmetry
    with respect to two different indefinite inner products, we prove
    the existence of nonnegative and nonpositive solutions of the
    Riccati equation.  Moreover, conditions for the boundedness 
    and uniqueness
    of these solutions are established.
  } \\[1.5ex]
  \parbox{11cm}{\small\textbf{Keywords.} Riccati equation,
    Hamiltonian, dichotomous, bisectorial, invariant subspace,
    $p$-subordinate perturbation.} \\[1.5ex]
  \parbox{11cm}{\small\textbf{Mathematics Subject Classification.}
    47A62, 47B44, 47N70.}
\end{center}
\vspace{0em}

\section{Introduction}

In this paper we prove the existence of solutions of algebraic Riccati equations
\begin{equation}\label{eq:ricc}
  A^*X+XA+XBX-C=0
\end{equation}
on a Hilbert space $H$ where all coefficients are unbounded linear operators and $B$, $C$ are nonnegative.
Riccati equations of this type, and in particular their nonnegative solutions,
are of central importance in systems theory,
see e.g.\ \cite{curtain-zwart,lancaster-rodman} and the references therein;
recently, the case of unbounded $B$ and $C$ has gained much attention 
\cite{lasiecka-triggiani,opmeer-curtain04,staffans05,weiss-weiss}.

The existence of solutions $X$ of the Riccati equation \eqref{eq:ricc} is intimately related to the  existence of 
graph subspaces $G(X) = \set{(u, Xu)}{u \in \mdef(X)}$ that are invariant under the associated 
\emph{Hamiltonian} 
\begin{equation}\label{eq:ham-intro}
T=\pmat{A&B\\C&-A^*}.
\end{equation}
Moreover, properties of a solution $X$ of \eqref{eq:ricc} such as selfadjointness, nonnegativity or boundedness 
can be characterised by properties of the corresponding graph subspace $G(X)$ with respect to certain indefinite inner products.

In the finite-dimensional case, the connection between solutions of
Riccati equations and invariant graph subspaces of Hamiltonians led to
an extensive description of all solutions, see
e.g.\ \cite{bittanti-laub-willems,lancaster-rodman}. In the
infinite-dimensional case, the existence of invariant subspaces is a
more subtle problem since the Hamiltonian $T$ is not normal.  If all
coefficients of the Riccati equation, and hence all entries of $T$,
are unbounded, the spectrum of the Hamiltonian may touch at infinity
and there are neither the spectral theorem nor Riesz projections
available to define invariant subspaces.

There are two different approaches to overcome these difficulties
which require different additional properties of the Hamiltonian $T$. In
\cite{kuiper-zwart,wyss-rinvsubham,wyss-unbctrlham} infinitely many
solutions of \eqref{eq:ricc} were constructed in the case that $T$ has
a Riesz basis of (possibly generalised) eigenvectors.  In
\cite{langer-ran-rotten,bubak-mee-ran} the existence of a nonnegative and a
nonpositive solution, and conditions for their boundedness, were
obtained in the case that $T$ is dichotomous and $B$, $C$ are bounded.

In the present paper, we prove the existence of solutions of the
Riccati equation \eqref{eq:ricc}, and characterise their properties,
without the assumptions that $T$ has a Riesz basis of generalised
eigenvectors or that $B$, $C$ are bounded. 

To this end we follow the dichotomy approach, but 
essentially new techniques are needed to establish the boundedness of solutions of the Riccati 
equation in the presence of unbounded $B$ and $C$.
In our main result
(Theorem~\ref{theo:ham-ricc}) we show that if $T$ is a nonnegative
diagonally $p$-dominant Hamiltonian (i.e.\ $B$, $C$ are nonnegative
and $p$-subordinate to $A^*$, $A$, respectively, with $p<1$), the
state operator $A$ is sectorially dichotomous, and
$\bigcap_{t\in\R}\ker (B(A^*+\ri t)^{-1})=\{0\}$, then there exists a
nonnegative solution $X_+$ and a nonpositive solution $X_-$ of the
Riccati equation \eqref{eq:ricc} or, more precisely, of
\begin{equation}\label{eq:riccati-intro}
    (A^*X_\pm+X_\pm(A+BX_\pm)-C)u=0, \qquad
    u\in\mdef(A)\cap X_\pm^{-1}\mdef(A^*).
\end{equation} 
In our second main result (Theorem~\ref{theo:bndricc}), we show that
if e.g.\ $A$ is sectorial with angle $\theta<\pi/2$, then the
nonnegative solution $X_+$ is bounded, 
uniquely determined
and \eqref{eq:riccati-intro} holds for all $u \in \mdef(A)$;
similarly, if $-A$ is sectorial with angle $\theta<\pi/2$, 
then $X_-$ is bounded and uniquely determined.

The assumption $\bigcap_{t\in\R}\ker (B(A^*\!+\ri t)^{-1})\!=\!\{0\}$
is trivially satisfied if $\ker{B}\!=\!\{0\}$. A necessary condition for it is
that $\ker B$ contains no eigenvectors of~$A^*$; if $A$ has a compact
resolvent and the system of generalised eigenvectors is complete, it
is also sufficient.  If $A$ generates a $C_0$-semigroup 
and $B$ is bounded, it is equivalent to the approximate controllability of the
pair $(A,B)$.

A novel ingredient of our approach are stability theorems for $p$-subordinate perturbations of sectorially dichotomous operators.
In brief, a linear operator $R$ on a Banach space $V$ is called \emph{$p$-subordinate} to a linear operator $S$ on $V$ with $p\in[0,1]$ if $\mdef(S) \subset \mdef(R)$ and there exists $c\ge0$ with
\[
  \|Ru\|\leq c \|u\|^{1-p}\|Su\|^p, \quad u\in\mdef(S);
\]
if $p<1$, this implies that $R$ is $S$-bounded with $S$-bound $0$. 
A linear operator $S$ on  $V$ is called
\emph{dichotomous} if
the spectrum $\sigma(S)$ has a gap along the imaginary axis $\ri \R$ and
there is a decomposition $V=V_+\oplus V_-$ into
$S$-invariant subspaces $V_\pm$ such
that the restrictions $S_+=S|_{V_+}$ and $S_-=S|_{V_-}$ have their
spectrum in the right and left half-plane, respectively; 
note that, even in the Hilbert space case, orthogonality is not assumed.
If $-S_+$ and $S_-$ are generators of exponentially decaying semigroups, then
$S$ is called \emph{exponentially dichotomous},
see~\cite{bart-gohberg-kaashoek}; if these semigroups are even
analytic, then $S$ is \emph{sectorially dichotomous}, see Section
\ref{sec:dichot} below.

%

The assumption that the state operator $A$ is sectorially dichotomous
implies that $A$ is \emph{bisectorial} (i.e.\ a bisector around
$\ri\R$ is contained in the resolvent set $\varrho(A)$ and
$\lambda(A-\lambda)^{-1}$ is uniformly bounded on this bisector).
Bisectorial operators play an important role in the study of maximal
regularity of evolution equations $u'+Au=f$ on $\R$, see
e.g.\ \cite{arendt-bu,arendt-duelli}.  
Exponentially dichotomous operators have a wide range of
applications, e.g.\ to Wiener-Hopf factorisation, see
\cite{bart-gohberg-kaashoek79,bart-gohberg-kaashoek,vdmee-book}.
The spectral decomposition of a
dichotomous Hamiltonian operator function may be used to show the
conditional reducibility of this operator function, see \cite{azizov-dijksma-gridneva12}.

%

In systems theory, e.g.\ for  systems with boundary control and observation, see \cite{wyss-unbctrlham},
the unbounded operators $B$ and $C$ need not have realisations as symmetric 
operators \emph{on} $H$ but, instead, map into an extrapolation space.
The results of this paper are a first step in this direction; the generalisation 
to Riccati equations involving extrapolation spaces is work in progress.

The article is organised as follows:
In Section~\ref{sec:dichot} we introduce sectorially dichotomous operators
and present some of their important properties.
In Section~\ref{sec:psub} we study the stability of bisectoriality and sectorial dichotomy under $p$-subordinate perturbations
and we investigate their effect on the spectrum.
In Section~\ref{sec:ham} we prove that a Hamiltonian \eqref{eq:ham-intro} with sectorially dichotomous $A$ and nonnegative $B$,~$C$ that are $p$-subordinate to $A^*$, $A$, respectively, is dichotomous. 
We employ the symmetry of $T$ with respect to two different indefinite inner
products $[\cdot|\cdot]_1$, $[\cdot|\cdot]_2$, used before in
\cite{kuiper-zwart}, \cite{langer-ran-temme}, \cite{langer-ran-rotten},
to show that the corresponding invariant subspaces $V_+$, $V_-$ are
hypermaximal neutral in $[\cdot|\cdot]_1$ and nonnegative, nonpositive,
respectively, in $[\cdot|\cdot]_2$.
In Section~\ref{sec:ricc} we exploit these properties to prove, in 
Theorem~\ref{theo:ham-ricc}, that $V_\pm$ are graphs or inverse graphs of operators $X_\pm$ and that $X_\pm$ are solutions of the Riccati equation \eqref{eq:ricc} if 
$\bigcap_{t\in\R}\ker (B(A^*+\ri t)^{-1})=\{0\}$. Moreover, we derive necessary as well as sufficient conditions for the latter assumption.
In Section~\ref{sec:bndsol} we prove, in Theorem \ref{theo:bndricc}, that
$X_+$ (or $X_-$) is bounded
and uniquely determined
provided that $A$ (or $-A$) is sectorial with
angle $\theta<\pi/2$.  Our proof exploits the continuous dependence of the
subspaces $V_\pm$, and hence of $X_\pm$, on $B$ and $C$, see Proposition
\ref{prop:contgraph}; it differs substantially from the one in
\cite{langer-ran-rotten} for bounded $B$,~$C$.
In the final Section~\ref{sec:example} we illustrate our theory by three
examples in which all entries of the Hamiltonian are partial differential
or unbounded multiplication operators; in all cases neither the results of
\cite{langer-ran-rotten,bubak-mee-ran} nor those of
\cite{kuiper-zwart,wyss-rinvsubham,wyss-unbctrlham} apply, either because $B$,
$C$ are unbounded or because the Hamiltonian does not have a Riesz basis of
generalised eigenvectors.

In this paper the following notation is used. For a closed linear
operator $T$ on a Banach space $V$ we denote the domain by $\mdef(T)$,
the kernel by $\ker (T)$, the spectrum by $\sigma(T)$, the point
spectrum by $\sigma_p(T)$, and the resolvent set by
$\varrho(T)$. Further, by $\C_+$ and $\C_-$ we denote the open right and
open left half-plane, respectively.

\section{Sectorially dichotomous operators}
\label{sec:dichot}

In this section we introduce and study sectorially dichotomous operators.
They form a subclass of exponentially dichotomous operators 
for which there exist invariant spectral subspaces corresponding to the spectral
parts in the left and the right half-plane, even if none of them is
bounded.

We begin by briefly recalling the notions of dichotomous and exponentially
dichotomous operators, see \cite{bart-gohberg-kaashoek,langer-ran-rotten}, and of sectorial and bisectorial operators, see~\cite{arendt-bu}.

\begin{definition}\label{def:dichot}
  A densely defined linear operator $S$ on a Banach space~$V$
  is called \emph{dichotomous} if there exist $h>0$ and
  complementary closed subspaces
  $V_+,V_-\subset V$, i.e.\ $V=V_+\oplus V_-$, such that
  \begin{itemize}
  \item[(i)] 
    $\bigset{z\in\C}{|\Real z|<h}\subset\varrho(S)$,
  \item[(ii)] $V_+$ and $V_-$ are $S$-invariant, i.e.\
    $S(\mdef(S)\cap V_\pm)\subset V_\pm$, and
  \item[(iii)] $\sigma(S|_{V_+})\subset\C_+$ {and}
    $\sigma(S|_{V_-})\subset \C_-$;
  \end{itemize}
  in this case, the maximal $h_0$ with (i) is called
  \emph{dichotomy gap} of $S$. A dichotomous operator is called 
  \emph{exponentially dichotomous} if
  \begin{itemize}
  \item[(iv)] 
  $-S|_{V_+}$ and $S|_{V_-}$ are generators of exponentially decaying semigroups.
  \end{itemize}
We call $V_\pm$ the spectral subspaces corresponding to the
dichotomous operator $S$; we write $S_\pm:=S|_{V_\pm}$ for the restrictions of
$S$ to $V_\pm$ and denote by $P_\pm$ the spectral projections onto $V_\pm$.
\end{definition}

Dichotomous operators admit a block diagonal matrix representation with respect to the 
decomposition $V=V_+\oplus V_-$ in the following sense:

\begin{definition}[\mbox{\cite[\S III.5.6]{kato}}]\label{def:decomp}
  Let $S$ be a linear operator on a Banach space~$V$ and $V_1,V_2\subset V$ 
  complementary closed subspaces.
  Then $S$ is said to  \emph{decompose} with respect to 
  the direct sum $V=V_1\oplus V_2$ if
  \begin{itemize}
  \item[(i)] $V_1$ and $V_2$ are $S$-invariant, and
  \item[(ii)] $\mdef(S)=(\mdef(S)\cap V_1)\oplus(\mdef(S)\cap V_2)$.
  \end{itemize}
\end{definition}

Note that even in the Hilbert space case it is not assumed that $V_1$
and $V_2$ are orthogonal, i.e.\ $V_1$ is not a reducing subspace of $S$
in the sense of \cite{akhiezer-glazman,weidmann80}.

\begin{remark}
  If $S$ decomposes with respect to $V=V_1\oplus V_2$, then $S$ admits the
  block operator matrix representation
  \[S=\pmat{S|_{V_1}&0\\0&S|_{V_2}};\]
  in particular, 
  \(\sigma(S)=\sigma(S|_{V_1})\cup\sigma(S|_{V_2})\)
  and,  for every $z\in\varrho(S)$, the subspaces
  $V_1$ and $V_2$ are also $(S-z)^{-1}$-invariant.
\end{remark}

\begin{lemma}\label{lem:dichot}
  If the linear operator $S$ is dichotomous, then it decomposes with
  respect to its spectral subspaces $V=V_+\oplus V_-$.
\end{lemma}
\begin{proof}
  We only have to verify  property (ii) in Definition~\ref{def:decomp}.
  The inclusion ``$\supset$'' is trivial.
  Let $x\in\mdef(S)$. Then $Sx=y_++y_-$ with $y_\pm\in V_\pm$.
  Since $0\in \varrho(S_{\pm})$ by condition (iii) in
  Definition~\ref{def:dichot}, 
   we can set $x_\pm:=(S_{\pm})^{-1}y_\pm \in \mdef(S)\cap V_\pm$
  and obtain
  \[S(x_++x_-)=S_{+}x_++S_{-}x_-=y_++y_-=Sx.\]
  Because $0\in\varrho(S)$, this implies that $x=x_++x_-\in(\mdef(S)\cap V_+)\oplus(\mdef(S)\cap V_-)$.
\end{proof}

\begin{remark}\label{rem:dichot}
  There are two simple cases in which condition (i) in Definition \ref{def:dichot},  
  \(\set{z\in\C}{|\Real z|<h}\subset\varrho(S)\),
  already suffices for the dichotomy of $S$:
  \begin{enumerate}
  \item if $S$ is a normal operator on a Hilbert space;
  \item if one of $\sigma_\pm(S)=\sigma(S)\cap\C_\pm$ is bounded.
  \end{enumerate}
  In the first case,  the existence of the subspaces $V_\pm$ is a consequence
  of the spectral theorem;
  in the second case, the  Riesz projection corresponding to the bounded
  part $\sigma_-(S)$ or $\sigma_+(S)$ of $\sigma(S)$ may be used to define $V_-$ or $V_+$, 
  compare \cite[\S III.6.4]{kato}.
\end{remark}

The following result is essential in characterising dichotomous
operators possessing the additional property that
the spectral projections are given by a resolvent integral along the imaginary axis;
its proof is based on an earlier deep result by Bart, Gohberg,
and Kaashoek, see \cite[Theorem~3.1]{bart-gohberg-kaashoek} and also
\cite[Theorem~XV.3.1]{gohberg-goldberg-kaashoek}.

\begin{theo}[\mbox{\cite[Theorem~1.1]{langer-tretter}}]
  \label{theo:dichot}
  Let $S$ be a closed densely defined linear operator on a Banach space~$V$ and
  $h>0$ such that
  \begin{itemize}
  \item[{\rm (i)}] $\set{z\in\C}{|\Real z|\leq h}\subset\varrho(S)$ and
    $\,\sup_{|\Real z|\leq h}\|(S-z)^{-1}\|<\infty$;
  \item[{\rm (ii)}] 
    $\lim_{|s|\to\infty}\sup_{r\in[0,h]}\|(S-r-\ri s)^{-1}\|=0$; \vspace{-1mm}
  \item[{\rm (iii)}] the Cauchy principal value at infinity
    $\displaystyle\int_{\ri\R}^\prime(S-z)^{-1}x\,dz$ exists for all $x\in V$.
  \end{itemize}
  Then $S$ is dichotomous and
  the corresponding projections $P_+$, $P_-$ satisfy
  \[\frac{1}{\pi\ri}\int_{\ri\R}^\prime(S-z)^{-1}x\,dz=P_+x-P_-x,
    \quad
    x\in V.\]
\end{theo}

\begin{remark}
A standard Neumann series argument shows that  assumptions
(i) and (ii) in Theorem~\ref{theo:dichot} are satisfied if
\[
\ri\R\subset\varrho(S) \ \mbox{  and } \
\lim_{|t|\to\infty}\|(S-\ri t)^{-1}\|=0.
\] %
\end{remark}

To obtain a sufficient condition for assumption (iii), 
we now introduce sectorially dichotomous
operators, which form a subclass of exponentially dichotomous
operators.
First we need the notion of sectorial and bisectorial
operators, see e.g.~\cite{arendt-bu}.

\begin{definition}\label{def:sect}
  Let $S$ be a densely defined linear operator on a Banach space.
  \begin{itemize}
  \item[{\rm (i)}] $S$ is called 
    \emph{sectorial with angle}\footnote{
          Throughout the article we use the conventions $-\pi<\arg z\leq\pi$
          and $\arg 0=0$ for the argument of a complex number.}
    $\theta\in[0,\pi[\,$ \emph{and radius} $r\geq0$ if\,
    \begin{equation}
    \label{sector}
    \sigma(S)\subset\Sigma_\theta\cup\overline{B_r(0)}
    \qquad\text{where}\qquad \Sigma_\theta:=\bigset{z\in\C}{|\arg
    z|\leq\theta}
    \end{equation} 
    and for every $\theta'\in\,]\theta,\pi]\,$ there
    exists $M>0$ such that
    \begin{equation}\label{eq:sect-resolv}
      \|(S-z)^{-1}\|\leq\frac{M}{|z|}, \qquad 
      |\arg z|\geq\theta',\; |z|>r;
    \end{equation}
    $S$ is called \emph{sectorial with angle} $\theta\in[0,\pi[\,$,
    or simply \emph{sectorial}, if $r=0$.
  \item[{\rm (ii)}]
    $S$ is called \emph{bisectorial with angle} $\theta\in[0,\pi/2[\,$
    \emph{and radius} $r\geq0$ if
    \[\sigma(S)\subset\Sigma_\theta\cup(-\Sigma_\theta)\cup\overline{B_r(0)}\]
    and for every $\theta'\in\,]\theta,\pi/2]$ 
    there exists $M>0$ such that
    \begin{equation}\label{eq:bisectest}
      \|(S-z)^{-1}\|\leq\frac{M}{|z|}, \qquad 
      \theta'\leq|\arg z|\leq\pi-\theta',\;|z|>r;
    \end{equation}
    $S$ is called \emph{bisectorial with angle} $\theta\in[0,\pi/2[\,$,
    or simply \emph{bisectorial}, if $r=0$.
  \end{itemize}
\end{definition}

The bisector on which the resolvent estimate \eqref{eq:bisectest} holds is denoted by,
see Fig.~\ref{fig:psubpert},
\begin{equation}\label{eq:bisect}
  \Omega_{\theta',r}:=\set{z\in\C}
	{\theta'\leq|\arg z|\leq\pi-\theta',\;|z|>r}.
\end{equation}

\begin{remark}\label{rem:sect}
  \begin{itemize}
  \item[(i)]
    In semigroup theory, often $-S$ instead of $S$ is
    called sectorial.
  \item[(ii)]
    $S$ is sectorial  with angle $\theta<\pi/2$
    if and only if $-S$ is the generator of a bounded analytic semigroup,
    see e.g.\ \cite[Theorem~II.4.6]{engel-nagel}.
 \item[(iii)]
    If $V$ is a Hilbert space with scalar product $(\cdot|\cdot)$ and
    \[
      W(S):=\bigset{(Sx|x)}{x\in\mdef(S),\,\|x\|=1} 
    \]
    is the \emph{numerical range} of $S$, then $S$
    is sectorial with angle $\theta\leq\pi/2$ 
    if 
    \[W(S)\subset \Sigma_\theta \quad\text{and}\quad
      \varrho(S)\setminus W(S)\neq\varnothing;\]
    in this case, for every $\theta'\in\,]\theta,\theta+\pi/2]$
    the estimate \eqref{eq:sect-resolv} holds with
    $M=(\sin(\theta'-\theta))^{-1}$ and $r=0$,
    compare  \cite[Theorem~V.3.2 and \S V.3.10]{kato}.
%
  \item[(iv)]
    If $S$ satisfies 
    \eqref{eq:sect-resolv}
    for some $\theta'\in\,]0,\pi]$, $r\geq0$ and $M>0$, 
    then there exists   $\theta\in[0,\theta'[$
    such that $S$ is sectorial with  angle $\theta$ and radius $r$;
    this follows from  a standard Neumann series argument.
    Similarly, if $S$ satisfies \eqref{eq:bisectest} for some
    $\theta'\in\,]0,\pi/2]$, $r\geq0$ and $M>0$, 
      then there exists $\theta\in[0,\theta'[\,$
    such that $S$ is bisectorial with angle $\theta$ and radius $r$.
  \item[(v)] 
    If $S$ is the direct sum of two operators $S_+$ and $S_-$
    where $S_+$ and $-S_-$ are sectorial with angle $\theta<\pi/2$ and
    radius $r\geq0$, then $S$ is bisectorial with angle $\theta$ and
    radius $r$, see the proof of Lemma~\ref{lem:sectdichot}~(ii)
    below.
  \end{itemize}
\end{remark}

\begin{definition}
\label{def:sectdichot}
  A dichotomous operator $S$ on a Banach space is called
  \emph{sectorially dichotomous with angle} $\theta\in[0,\pi/2[\,$
  if  $S_+$ and $-S_-$ are sectorial with angle $\theta$.
\end{definition}

\begin{remark}
  \begin{itemize}
  \item[(i)] The operator $S$ is sectorially dichotomous if and only
    if it is exponentially dichotomous and the exponentially decaying
    semigroups generated by $-S_+$ and $S_-$ are analytic.
  \item[(ii)]
    A simple example for an operator that is exponentially, 
    but not sectorially dichotomous, is a normal operator with discrete
    spectrum and eigenvalues
    $1+\ri k$ and $-1+\ri k$, $k\in\N$.
  \end{itemize}
\end{remark}

The next lemma shows that sectorially dichotomous operators are
bisectorial (compare Figure~\ref{fig:sectdichot}) and satisfy
condition~(iii) in Theorem~\ref{theo:dichot}.  

\begin{lemma}\label{lem:sectdichot}
  Let $S$ be sectorially dichotomous with angle $\theta \in [0,\pi/2[$
  and dicho\-tomy gap $h_0>0$. Then
  \begin{itemize}
  \item[{\rm (i)}]
  $
  \sigma(S)\subset\bigset{z\in\Sigma_\theta\cup(-\Sigma_\theta)}
	{|\Real z|\geq h_0}
  $;
  \item[{\rm (ii)}]
    $S$ is bisectorial with angle $\theta$;
  \item[{\rm (iii)}]
  the spectral projections $P_+$, $P_-$  corresponding to $S$ satisfy
  \vspace{-1mm}
  \begin{equation}\label{eq:sectdichot-int}
    \frac{1}{\pi \ri}\int_{\ri\R}^\prime(S-z)^{-1}x\,dz=P_+x-P_-x,
    \quad x\in V.
    \vspace{-1mm}
  \end{equation}
  \end{itemize}
\end{lemma}

\begin{figure}[h]
  \begin{center}
    \includegraphics{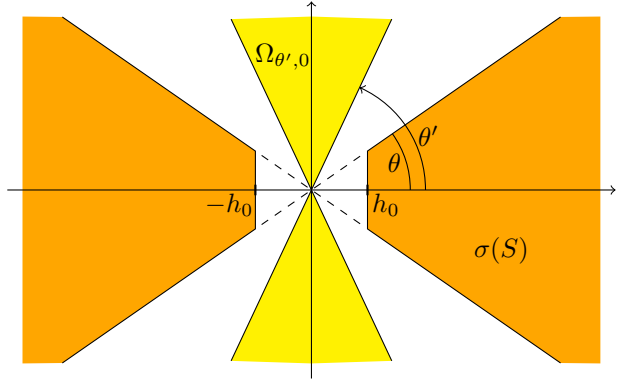} 
 \end{center}
  \caption{Situation in the proof of Lemma~\ref{lem:sectdichot}\,(ii) for
    sectorially dichotomous~$S$.}
  \label{fig:sectdichot}
\end{figure}

\begin{proof}
  (i) The claim is immediate from Definitions \ref{def:dichot}
  and~\ref{def:sect} because $\sigma(S)=\sigma(S_+)\cup\sigma(S_-)$.
  
  (ii)
  Let $\theta'\in\,]\theta,\pi/2]$.
  Since $S_+$ and $-S_-$ are sectorial with 
  angle $\theta$, there exist $M_\pm>0$ such that for
  $|\arg z|\geq\theta'$ we have
  $\|(\pm S_\pm-z)^{-1}\|\leq M_\pm/|z|$.
  For $z\in\Omega_{\theta',0}$
  we thus obtain, with \vspace{-1mm} $M:=M_+\|P_+\|+M_-\|P_-\|$,
  \[
    \|(S-z)^{-1}\|\leq\|(S_+-z)^{-1}P_+\|+\|(S_--z)^{-1}P_-\|
    \leq\frac{M}{|z|}.
  \vspace{-1mm}
  \]
  
  (iii)
  Since $S_+$ and $-S_-$ are sectorial with angle  $\theta<\pi/2$ and $0\in\varrho(S_\pm)$,
  \cite[Lemma~6.1]{langer-ran-rotten} implies that
  \[
    \int_{\ri\R}^\prime(\pm S_\pm-z)^{-1}x\,dz=\ri\pi x, \quad    x\in V_\pm.
  \vspace{-3mm}
  \]
  Consequently,
  \[\int_{\ri\R}^\prime(S-z)^{-1}x\,dz
    =\int_{\ri\R}^\prime(S_+-z)^{-1}P_+x\,dz
    +\int_{\ri\R}^\prime(S_--z)^{-1}P_-x\,dz
    =\ri\pi P_+x-\ri\pi P_-x \]
  for all $x\in V$.
\end{proof}

\begin{remark}
  Not every bisectorial operator with $0\in\varrho(S)$
  is sectorially dichotomous,
  see \cite[Theorem~3]{mcintosh-yagi} for a counter-example;
  note that hence the second implication of
  \cite[Proposition~1.8]{vdmee-book} does not hold.
  The question whether a bisectorial \emph{and} dichotomous operator $S$ is
  sectorially dichotomous will be considered in a forthcoming paper;
  while we know that the restrictions $S_+$ and $-S_-$ have their
  spectrum in a sector $\Sigma_\theta$ and satisfy resolvent estimates
  on $\Omega_{\theta,0}$, it is not clear that these estimates also
  hold on the left half-plane, 
  as required for sectoriality.
\end{remark}

In Section \ref{sec:ham} below we will consider Hamiltonians whose state operator $A$ is sectorially
dichotomous; in systems theory $A$ is usually even assumed to generate a strongly continuous semigroup. 
The following lemma characterises this situation.

\begin{lemma}\label{lem:analytsemigrp}
  For a linear operator $S$ in a Banach space the following are equivalent:
  \begin{itemize}
  \item[{\rm (i)}] $S$ is sectorially dichotomous and generates a strongly
    continuous semigroup;
  \item[{\rm (ii)}] $S$ is sectorially dichotomous with bounded $S_+$;
  \item[{\rm (iii)}] $S$ generates a $($not necessarily bounded$\,)$ analytic
    semigroup and $\ri\R\subset\varrho(S)$.
  \end{itemize}
\end{lemma}

\begin{proof}
  (i)$\Rightarrow$(iii):
  Since $S$ generates a strongly continuous semigroup, there exist
  $M>0$ and $\omega\in\R$ such that
  \[\|(S-z)^{-1}\|\leq\frac{M}{\Real z-\omega},\qquad \Real z>\omega.\]
  Together with \eqref{eq:bisectest}, 
  it is not difficult to conclude that there exist
  $M'>0$, $\omega'>\omega$ and $\varphi>\pi/2$ such that
  \begin{equation}\label{eq:anasemcond}
    \|(S-z)^{-1}\|\leq\frac{M'}{|z-\omega'|}\quad\text{for}\quad
    |\arg(z-\omega')|\leq\varphi.
  \end{equation}
  Hence $S$ generates an analytic semigroup.

  (iii)$\Rightarrow$(ii):
  Since $S$ generates an analytic semigroup,
  it satisfies an estimate~\eqref{eq:anasemcond}.
  Together with the assumption
  $\ri\R\subset\varrho(S)$ this implies
  that
  the part $\sigma_+(S)$ of the spectrum in the right half-plane
  is bounded and hence $S$ is dichotomous with bounded $S_+$, see 
  Remark~\ref{rem:dichot}; in particular, $S_+$ is sectorial with angle
  less than $\pi/2$.
  By \eqref{eq:anasemcond}, 
  also $-S_-$ is sectorial with angle less than $\pi/2$ and thus
  $S$ is sectorially dichotomous.

  (ii)$\Rightarrow$(i): Since $S_+$ is bounded,  it generates
  a strongly continuous semigroup. Due to the sectorial dichotomy, the same
  is true for $S_-$ and hence also for~$S$.
\end{proof}

Next we show that the adjoint $S^*$ of a
sectorially dichotomous operator $S$ on a Hilbert space $H$ is again
sectorially dichotomous.
The difficulty here is that the spectral decomposition $H=H_+\oplus H_-$
of $S$ is not necessarily orthogonal; for the simpler
orthogonal case, compare \cite[Exercise~5.39]{weidmann80}.

\begin{lemma}
  Let $S$ be a closed densely defined linear operator on a Hilbert space $H$ that
  decomposes with respect to a $($not necessarily orthogonal$\,)$
  direct sum $H=H_1\oplus H_2$.
  Then $S^*$ decomposes with respect to $H=H_2^\perp\oplus H_1^\perp$,
  and we have\footnote{
    We denote the complex conjugate of a set $G\subset\C$ by 
    $G^*=\set{\bar{z}\in\C}{z\in G}$.}
  \begin{align*}
    \sigma(S^*|_{H_2^\perp})=\sigma(S|_{H_1})^*, \quad
    \sigma(S^*|_{H_1^\perp})=\sigma(S|_{H_2})^*;
  \end{align*}
  moreover, if $P_{1/2}$ are the projections onto $H_{1/2}$ associated with
  $H=H_1\oplus H_2$, then
   \begin{align*}  
    \|(S^*|_{H_{2}^\perp}-\bar{z})^{-1}\| &\leq \|P_1\|\,\|(S|_{H_{1}}-z)^{-1}\|, 
    \\
    \|(S^*|_{H_{1}^\perp}-\bar{z})^{-1}\| &\leq \|P_2\|\,\|(S|_{H_{2}}-z)^{-1}\|.
   \end{align*} 
\end{lemma}

\begin{proof}
  We have $I=P_1+P_2$ and $\range(P_j)=H_j$. 
  Hence $P_1^*$, $P_2^*$ are projections
  with  $I=P_1^*+P_2^*$ and
  \[\range(P_{1}^*)=\ker P_{2}^*=\range(P_{2})^\perp=H_{2}^\perp, \quad \range(P_2^*)=H_1^\perp;\]
  in particular, $H=H_2^\perp\oplus H_1^\perp$.

  To show that $H_1^\perp$ is $S^*$-invariant, let 
  $y\in H_1^\perp\cap\mdef(S^*)$ and $x\in H_1\cap\mdef(S)$.
  Then $(S^*y|x)=(y|Sx)=0$ since $Sx\in H_1$.
  Because $\mdef(S)\subset H$ is dense, $H_1\cap\mdef(S)\subset H_1$ is 
  dense, too, and we obtain $S^*y\in H_1^\perp$. Similarly, one can show that
  $H_2^\perp$ is $S^*$-invariant.
  Now let $y\in\mdef(S^*)$ and $x\in\mdef(S)$. 
  Since $(SP_1)^* \supset P_1^*S^*$ and hence $(SP_1)^*|_{\mdef(S^*)} = P_1^*S^*$, we have that
  \[(P_1^*y|Sx)=(y|P_1Sx)=(y|SP_1x)=((SP_1)^*x|x)=(P_1^*S^*y|x)\]
  is bounded in $x$
  and thus $P_1^*y\in\mdef(S^*)$. 
  This implies that $P_2^*y=y-P_1^*y \in \mdef(S^*)$. Thus
  \[\mdef(S^*)=(\mdef(S^*)\cap\range(P_1^*))
    \oplus(\mdef(S^*)\cap\range(P_2^*)),\]
  and so $S^*$  decomposes with respect to
  $H=H_2^\perp\oplus H_1^\perp$.

  Finally, let $z\!\in\!\varrho(S|_{H_1})$ and 
  $R_1:=(S|_{H_1}-z)^{-1}P_1\!:H\!\to H$.  Then
  \begin{alignat*}{2}
    &P_1x=R_1(S-z)x,\qquad &&x\in\mdef(S),
    \\
    &P_1x=(S-z)R_1x,&&x\in H.
  \end{alignat*}
  We have $(R_1(S-z))^*=(S^*-\bar z)R_1^*$ because $R_1$ is bounded
  and $((S-z)R_1)^*|_{\mdef(S^*)} = R_1^*(S^*-\bar z)$. Hence we
  obtain
  \begin{alignat*}{2}
    &P_1^*y=(S^*-\bar{z})R_1^*y,\qquad &&y\in H,
    \\
    &P_1^*y=R_1^*(S^*-\bar{z})y,&&y\in\mdef(S^*).
  \end{alignat*}
  Since 
  \(\overline{\range(R_1^*)}=(\ker R_1)^\perp
    =(\ker P_1)^\perp=H_2^\perp=\range(P_1^*)\),
  this yields
  \begin{alignat*}{2}
    &y=(S^*|_{H_2^\perp}-\bar{z})R_1^*y, \qquad&&y\in H_2^\perp,
    \\
    &y=R_1^*(S^*|_{H_2^\perp}-\bar{z})y,&& y\in\mdef(S^*)\cap H_2^\perp.
  \end{alignat*}
  Consequently, $\bar{z}\in\varrho(S^*|_{H_2^\perp})$ with
  $(S^*|_{H_2^\perp}-\bar{z})^{-1}=R_1^*|_{H_2^\perp}$.  
  Exchanging the roles
  of $S$ and $S^*$ as well as those of $H_1$ and $H_2$, we obtain
  $\varrho(S|_{H_{1/2}})=\varrho(S^*|_{H_{2/1}^\perp})^*$ and
  \[\|(S^*|_{H_{2/1}^\perp}-\bar{z})^{-1}\|
    \leq\|R_{1/2}^*\|=\|R_{1/2}\|\leq
    \|P_{1/2}\|\,\|(S|_{H_{1/2}}-z)^{-1}\|.
  \vspace{-3mm}  
  \]
  \qedup
\end{proof}

\vspace{0.5mm}

\begin{coroll}\label{coroll:adjsectdichot}
  If\, $S$ is a sectorially dichotomous operator with angle 
  $\theta\,\in[0,\pi/2[$ on a
  Hilbert space, 
  then the adjoint $S^*$ is also sectorially dichotomous with angle~$\theta$.
\end{coroll}

\begin{proof}
  Let $H=H_+\oplus H_-$ be the decomposition corresponding to $S$
  and $h_0>0$~the dichotomy gap of $S$.
  Then $S^*$ decomposes with respect to $H=H_-^\perp\oplus H_+^\perp$.~Moreover,
  \[
    \sigma(S^*|_{H_-^\perp})=\sigma(S|_{H_+})^*
      \subset\bigset{z\in\Sigma_\theta}{\Real z\geq h_0}, \]
  and for $\theta'>\theta$ there exists $M>0$ such that
  \[\|(S^*|_{H_-^\perp}-z)^{-1}\|
      \leq \|P_+\|\|(S|_{H_+}-\bar{z})^{-1}\|\leq\frac{M\|P_+\|}{|z|},
      \quad|\arg z|\geq\theta'.\]
  An analogous reasoning applies to $-S^*|_{H_+^\perp}$,
  and we conclude that $S^*$ is sectorially dichotomous with angle $\theta$.
\end{proof}

\section{{\boldmath $p$}-subordinate perturbations}
\label{sec:psub}

In this section we show that bisectoriality is stable under
$p$-subordinate perturbations and that $p$-subordinate perturbations
of sectorially dichotomous operators are still dichotomous.
To begin with, we briefly recall 
the concept of $p$-subordinate perturbations which was studied e.g.\ in
\cite[\S I.7.1]{krein} and \cite[\S5]{markus}.

\begin{definition}
  Let $S,R$ be linear operators on a Banach space.
  \begin{itemize}
  \item[(i)] $R$ is called \emph{relatively bounded} with respect to
    $S$ or \emph{$S$-bounded} if $\mdef(S)\subset\mdef(R)$ and there 
    exist $a,b\geq0$ such that    
    \begin{equation}\label{eq:relbnd}
      \|Rx\|\leq a\|x\|+b\|Sx\|, \quad
       x\in\mdef(S);
    \end{equation}
    the infimum of all $b$ such that \eqref{eq:relbnd} holds with some
    $a\geq 0$ is called the \emph{relative bound} of $R$ 
    with respect to $S$ or \emph{$S$-bound} of $R$.
  \item[(ii)]
    $R$ is called \emph{$p$-subordinate} to $S$ with $p\in[0,1]$ 
    if $\mdef(S)\subset\mdef(R)$
    and there exists $c\geq0$ such that 
    \begin{equation}\label{eq:psub}
      \|Rx\|\leq c\|x\|^{1-p}\|Sx\|^p, \quad
      x\in\mdef(S);
    \end{equation}
    the minimal constant $c$ such that \eqref{eq:psub} holds is called
    the \emph{$p$-subordination bound} of $R$ to $S$.
  \end{itemize}
\end{definition}

Note that, in contrast to the relative bound, the infimum over all $c$ that satisfy \eqref{eq:psub} does itself satisfy \eqref{eq:psub} and
hence the $p$-subordination bound is indeed a minimum.

\begin{lemma}\label{lem:psub}
  Let $S$, $R$ be linear operators satisfying  $\mdef(S)\subset\mdef(R)$ and let $p\in[0,1]$.
\begin{itemize}
\item[{\rm (i)}]
  $R$ is $p$-subordinate to $S$ if and only if
  there exists a constant
  $c'\geq0$ such that
  \begin{equation}\label{eq:psubequiv}
    \|Rx\|\leq c'(\eps^{-p}\|x\|+\eps^{1-p}\|Sx\|), \quad
    x\in\mdef(S),\ \eps>0.
  \end{equation}
\item[{\rm (ii)}]
  If\, $R$ is $p$-subordinate to $S$ with $p<1$, then $R$
  is $S$-bounded with $S$-bound $0$.
\item[{\rm (iii)}]
  If\, $0\in\varrho(S)$ and $R$ is $p$-subordinate to $S$, then 
  $R$ is $q$-subordinate to $S$ for every $q>p$.
\end{itemize}
\end{lemma}
\begin{proof}
  (i) was proved in \cite[page 146]{krein}, (ii) follows from (i),
  and (iii) is a consequence of the inequality
  \(\|x\|^{1-p}\leq\|x\|^{1-q}\|S^{-1}\|^{q-p}\|Sx\|^{q-p}\), 
  $x\in \mdef(S)$.
\end{proof}

The following lemma provides conditions guaranteeing that e.g.\ a
multiplication operator $R$ in $L^q(\Omega)$ with $q\in[1,\infty[$ and
open $\Omega\subset \R^n$ is $p$-subordinate to an elliptic partial
differential operator $S$ of order $m>0$; more generally, $R$ may
also be a partial differential operator of order $k\le m$.

In fact, if $W^{m,q}(\Omega)$ denotes the usual Sobolev
space of $m$ times weakly differentiable functions with derivatives
in $L^q(\Omega)$, then 
we consider operators $S$ on $L^q(\Omega)$ such that 
$\mdef(S)\subset W^{m,q}(\Omega)$ and $S$ satisfies a so-called 
\emph{a priori estimate},
\begin{equation}\label{eq:apriori}
  \|u\|_{W^{m,q}(\Omega)}\leq c_0(\|u\|_{L^q(\Omega)}+\|Su\|_{L^q(\Omega)}),
  \quad u\in\mdef(S),
\end{equation}
with some constant $c_0>0$; such an estimate holds e.g.\ if $S$ is a
properly elliptic partial differential operator of order $m$ with
appropriate boundary conditions,
see \cite[Chap.~2, \S5]{lions-magenes}, 
\cite[\S5.3]{triebel}.

\begin{lemma}\label{lem:multop}
  Let $S$ be a linear operator on $L^q(\Omega)$, $q\in[1,\infty[$,
  such that $0\in\varrho(S)$, $\mdef(S)\subset W^{m,q}(\Omega)$ and an
  a priori estimate \eqref{eq:apriori} holds.
  \begin{itemize}
  \item[{\rm (i)}] Let $\Omega=\R^n$,
    $g:\R^n\to\C$ a locally integrable function, and 
    let $R$ be the corresponding $($maximal$)$ multiplication operator, 
  \[Ru:=gu,\qquad \mdef(R):=\set{u\in L^q(\R^n)}{gu\in L^q(\R^n)}.\]
  If there exist $s\in[0,n]$ and $c_1>0$ such that
  \begin{equation}\label{eq:multopcond}
    \int_{B_r(x_0)}|g(x)|^q\,dx\leq c_1r^s,
    \quad
    x_0\in\R^n,\ 0<r<1,
  \end{equation}
  and if 
  \[
    \left\{ \begin{array}{ll} 
       s>n-mq & \mbox{ if } \ q>1, \\
       s\geq n-m & \mbox{ if } \ q=1, 
       \end{array} \right.
  \]
  then $R$ is $p$-subordinate to $S$ with $\displaystyle p=\frac{1}{mq}(n-s)$.
  \item[{\rm (ii)}]
  If $\,R$ is a partial differential operator on $L^q(\Omega)$ of order 
  $k\leq m$ with
  coefficients in $L^\infty(\Omega)$, then $R$ is $\dfrac km$-subordinate to $S$. 
  \end{itemize}
\end{lemma}

\begin{proof}
  (i) Consider the measure $\mu$ on $\R^n$ given by
  \(\mu(A)=\int_A|g(x)|^q\,dx\). Then $\|gu\|_{L^q(\R^n)}=\|u\|_{L^q(\mu)}$.
  Due to assumption  \eqref{eq:multopcond}, we can apply
  \cite[\S1.4.7, Corollary~1]{mazya}
  and estimate 
  \begin{equation} 
  \label{eq:multsub}
  \|u\|_{L^q(\mu)}\leq c_2\|u\|_{W^{m,q}(\R^n)}^p\|u\|^{1-p}_{L^q(\R^n)},
  \quad u\in W^{m,q}(\R^n),
  \end{equation}
  with some constant $c_2>0$.
  The estimate \eqref{eq:apriori} together with $0\in\varrho(S)$ implies that
  \[\|u\|_{W^{m,q}(\R^n)}\leq c_0(1+\|S^{-1}\|)\|Su\|_{L^q(\R^n)},
   \quad u\in \mdef(S),\]
  and hence the subordination inequality \eqref{eq:psub} follows.
  
  (ii) The proof of (ii) is similar to that of (i) if, instead of
  \eqref{eq:multsub}, we use the 
  interpolation  inequality for intermediate derivatives, see
   \cite[Theorem~5.2]{adams-fournier},
   \begin{align*}
     \|u\|_{W^{k,q}(\Omega)}
    \leq c\|u\|_{W^{m,q}(\Omega)}^{k/m}\|u\|_{L^q(\Omega)}^{1-k/m},
    \quad u\in W^{k,q}(\Omega). \\[-13mm]
    \end{align*}
\end{proof}

\begin{remark}
  The subordination property in (ii) was used e.g.\ in 
  \cite[\S10]{markus} and \cite{wyss-psubpert,wyss-rinvsubham} 
  to obtain expansions in eigenfunctions of $S+R$. 
\end{remark}



Next we show that bisectoriality is stable under $p$-subordinate perturbations and we study their effect on the spectrum, see Figure \ref{fig:psubpert}.

\begin{lemma}\label{lem:psubpert}
  Let $S$, $R$ be linear operators,
  $S$ bisectorial with angle $\theta \in [0, \pi/2[$ 
  and radius $r\geq0$, and $R$ $p$-subordinate to $S$ with $p\in[0,1]$.
  \begin{itemize}
  \item[{\rm (i)}] 
    For every $\theta'\in\,]\theta,\pi/2]$
    there exists $M'\geq0$ such that
    \begin{equation}\label{eq:psubest}
      \|R(S-z)^{-1}\|\leq\frac{M'}{|z|^{1-p}}, \quad 
      z\in\Omega_{\theta',r},
    \end{equation}
    where $\Omega_{\theta',r}\!=\!\set{z\in\C} {\theta'\leq|\arg z|\leq\pi\!-\!\theta',\;|z|>r}$, see \eqref{eq:bisect}.
  \item[{\rm (ii)}]
    If $R$ is  even $p$-subordinate to $S$ with $p<1$, then for
    every $\theta'\in\,]\theta,\pi/2[$
    there exists
    $r'\geq r$ such that $S+R$ is bisectorial with angle $\theta'$ and 
    radius $r'$.
  \end{itemize}
\end{lemma}

\begin{figure}[h]
  \begin{center}
    \includegraphics{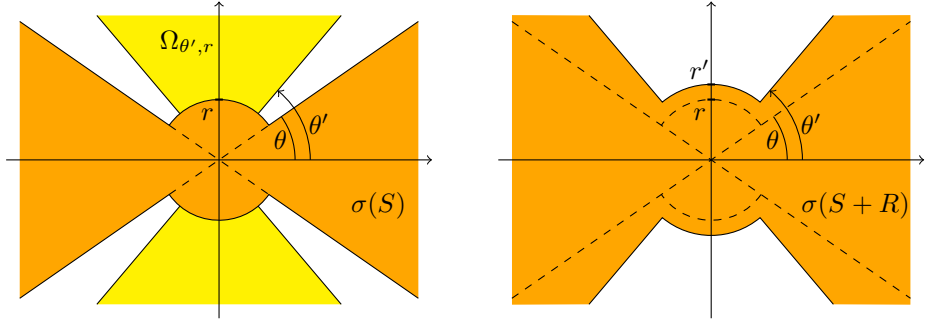} 
 \end{center}
  \caption{The perturbation of the spectrum of a bisectorial operator
    in Lemma~\ref{lem:psubpert}.}
  \label{fig:psubpert}
\end{figure}

\begin{proof}
  (i) Let $\theta'\in\,]\theta,\pi/2]$. Then, for
  $z\in\Omega_{\theta',r}$, we use \eqref{eq:bisectest} to estimate
  \[\|S(S-z)^{-1}\|\leq 1+|z|\cdot\|(S-z)^{-1}\|\leq 1+M.\]
  Hence, if $R$ is $p$-subordinate to $S$, then
  \[\|R(S-z)^{-1}\|\leq c\|(S-z)^{-1}\|^{1-p}\|S(S-z)^{-1}\|^p
    \leq c\cdot\biggl(\frac{M}{|z|}\biggr)^{1-p}(1+M)^p.\]

  (ii)  Let  $\theta'\in\,]\theta,\pi/2[\,$.
  By \eqref{eq:psubest} there exists $r'\geq r$ such that
  $\|R(S-z)^{-1}\|\leq1/2$ for all $z\in\Omega_{\theta',r'}$.
  A  Neumann series argument then implies that $z\in\varrho(S+R)$ for 
  $z\in\Omega_{\theta',r'}$, 
  \[(S+R-z)^{-1}=(S-z)^{-1}\bigl(I+R(S-z)^{-1}\bigr)^{-1},\]
  and $\|(S+R-z)^{-1}\|\leq2\|(S-z)^{-1}\|$. This
  and \eqref{eq:bisectest} imply that $S+R$ is bisectorial with
  angle $\theta'$ and radius $r'$.
\end{proof}

\begin{remark}
If the unperturbed operator $S$ in Lemma \ref{lem:psubpert} is
selfadjoint, and hence bisectorial with angle $\theta=0$, then the spectral
inclusion implied by Lemma \ref{lem:psubpert}~(ii) and displayed in
Figure \ref{fig:psubpert} follows from the spectral enclosure
\cite[Corollary~2.4]{cuenin-tretter} since $p$-subordinate
perturbations with $p<1$ have relative bound $0$.
\end{remark}

For bisectorial operators with radius $r=0$, the estimate
\eqref{eq:psubest} is, in fact, an equivalent characterisation of
$p$-subordinacy:

\begin{lemma}\label{lem:psubequiv}
  Let $S$, $R$ be linear operators,  
  $S$ bisectorial with angle $\theta \in [0, \pi/2[$ 
  and radius $r\geq0$, $\mdef(S)\subset\mdef(R)$, and $p\in[0,1]$.
  If there exist $\theta'\in\,]\theta,\pi/2]$ and $M'\geq0$
  with
  \begin{equation}\label{eq:RS-est}
    \|R(S-z)^{-1}\|\leq\frac{M'}{|z|^{1-p}}, \quad 
    z\in\Omega_{\theta',0},
  \end{equation}
  where $\Omega_{\theta',0}\!=\!\set{z\in\C} {\theta'\leq|\arg z|\leq\pi\!-\!\theta',\;|z|>0}$, then $R$ is $p$-subordinate to $S$.
\end{lemma}

\begin{proof}
  The estimate \eqref{eq:RS-est} 
  implies that
  \[\|Rx\|\leq\frac{M'}{|z|^{1-p}}\|(S-z)x\|
    \leq M'\bigl(|z|^p\|x\|+|z|^{p-1}\|Sx\|\bigr),
    \quad x\in\mdef(S).\]
  Choosing $z=\ri\eps^{-1}$, $\eps>0$, 
  we obtain \eqref{eq:psubequiv}; Lemma \ref{lem:psub}~(i) thus yields the claim.
\end{proof}

\begin{remark}\label{rem:psub}
  A result analogous to Lemma~\ref{lem:psubpert} holds for any subset
  $\Omega\subset\varrho(S)\setminus\{0\}$ such that there is
  an estimate \eqref{eq:bisectest} on $\Omega$
  instead of $\Omega_{\theta',r}$; in this case
  $S$ is not required to be bisectorial. In the same way  
  Lemma~\ref{lem:psubequiv} can be generalised 
  if, in addition, $\Omega$ satisfies
  the condition $\set{|z|}{z\in\Omega}=\R_+$.
\end{remark}


The following theorem on $p$-subordinate perturbations of 
dichotomous bisectorial operators is crucial for the next sections.
Compared to \cite[Theorem~1.3]{langer-tretter} we use $p$-subordinacy rather than an estimate of type \eqref{eq:psubest} 
and we only assume that the imaginary axis belongs to the set of points of regular type of the perturbed operator, not to its resolvent set.


Recall that for a linear operator $S$ on a Banach space, 
$z\in\C$ is called a
\emph{point of regular type} if there exists $c>0$ such that
\[\|(S-z)x\|\geq c\|x\|, \quad
x\in\mdef(S).\]
The set $r(S)$ of all points of regular type is open and satisfies
$\varrho(S)\subset r(S)$. If
$\Omega\subset r(T)$ is a connected subset such that 
$\Omega\cap\varrho(S)\neq\varnothing$, then $\Omega\subset\varrho(S)$,
see \cite[\S 78]{akhiezer-glazman}.
The complement $\C\setminus r(S)$ is the approximate point spectrum of $S$.

\begin{theo}\label{theo:dichotpert}
  Let $S$ be a closed densely defined linear operator on a Banach space~$V$
  such that
  \begin{itemize}
  \item[{\rm (i)}] $\ri\R\subset\varrho(S)$;
  \item[{\rm (ii)}] $S$ is bisectorial with angle $\theta\in[0,\pi/2[$ and
    radius $r\geq0$;
    \vspace{-2mm}
  \item[{\rm (iii)}] the integral
    \(\displaystyle\int_{\ri\R}'(S-z)^{-1}x\,dz\)
    exists for all $x\in V$.
  \end{itemize}
  Let $R$ be $p$-subordinate to $S$ with $p<1$.  If\,
  $\ri\R\subset r(S+R)$, then $S+R$ is dichotomous with dichotomy
  gap $h>0$, the corresponding projections $P_\pm$ satisfy
  \begin{equation}\label{eq:intproj}
    \frac{1}{\pi\ri}\int_{\ri\R}'(S+R-z)^{-1}x\,dz=P_+x-P_-x, \quad x\in V,
  \end{equation}
  and $S+R$ is bisectorial with some angle $\theta''\in\,]\theta,\pi/2[$.
  Moreover, for every
  $\theta'\in\,]\theta,\pi/2[$ there exists $r'\geq r$ such that $S+R$ 
  is also bisectorial with angle $\theta'$ and radius $r'$ and
  $($see Figure~{\rm \ref{fig:dichotpert}}$)$
  \begin{align}
    &\bigset{z\in\C}{|\Real z|< h}\cup\Omega_{\theta'',0}
    \cup\Omega_{\theta',r'} \subset\varrho(S+R).
    \label{eq:resolvincl}
  \end{align}
\end{theo}

\begin{figure}[h]
  \begin{center}
    \includegraphics{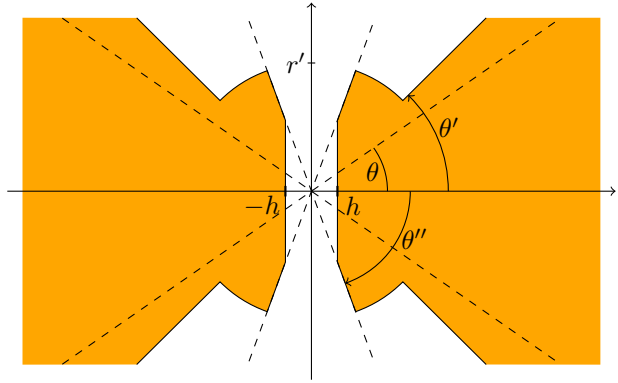}
  \end{center}
  \caption{Region containing $\sigma(S+R)$ in 
    Theorem~\ref{theo:dichotpert}.}
  \label{fig:dichotpert}
\end{figure}

\begin{proof}
  Lemma~\ref{lem:psubpert} implies the bisectoriality
  with angle $\theta'$ and radius $r'$.
  In particular, the connected subset $\ri\R$ of $r(S+R)$ contains 
  points from $\varrho(S+R)$ and thus
  $\ri\R\subset\varrho(S+R)$.
  Since $\varrho(S+R)$ is open and $(S+R-z)^{-1}$ is uniformly bounded
  on compact subsets,
  there exist $h>0$, $\theta''\in\,]\theta,\pi/2[$ such that
  $S+R$ is bisectorial with angle $\theta''$ and 
  \eqref{eq:resolvincl} holds.
  Consequently, $S+R$ satisfies the assumptions (i) and (ii) in
  Theorem~\ref{theo:dichot}.
  Furthermore, \eqref{eq:psubest} and the estimate \eqref{eq:bisectest}
  for $S+R$ imply that
  \[\int_{\ri\R}(S+R-z)^{-1}R(S-z)^{-1}\,dz\]
  exists in the uniform operator topology.
  From  the resolvent identity
  \[(S+R-z)^{-1}=(S-z)^{-1}-(S+R-z)^{-1}R(S-z)^{-1},
    \quad z\in \ri\R,\]
  we conclude that $S+R$ also satisfies assumption (iii)  in 
  Theorem~\ref{theo:dichot}.
\end{proof}

In view of Lemma~\ref{lem:sectdichot},
the previous result immediately applies to sectorially dichotomous operators.

\begin{coroll}\label{coroll:sectdichotpert}
  Let $S$ be sectorially dichotomous with angle $\theta\in[0,\pi/2[$.
  Let $R$ be $p$-subor\-di\-nate to $S$ with $p<1$ and
  $\ri\R\subset r(S+R)$.
  Then $S+R$ is dichotomous and all assertions of
  Theorem~{\rm\ref{theo:dichotpert}}  hold.
\end{coroll}


\section{Dichotomous Hamiltonians}
\label{sec:ham}

Hamiltonian operator matrices are block operator matrices of a particular form. Block operator matrices can be classified according to the domains of their entries into diagonally dominant, off-diagonally dominant, and top dominant, see \cite{tretter2,MR2514061,tretter-book}. Here we introduce the new class of diagonally $p$-dominant block operator matrices. 

\begin{definition}
Let $H_1$, $H_2$ be Hilbert spaces, consider densely defined linear operators $A$ in $H_1$, $B$ from $H_2$ to $H_1$, $C$ from $H_1$ to $H_2$, and $D$ in $H_2$, and let $p\in[0,1]$. Then the block operator matrix
\[  
  T=\pmat{A&B\\C&D} \quad \text{in} \ H_1 \times H_2
\]  
is called 
  \begin{itemize}
  \item[(i)\!] \hspace*{-0.5mm}\emph{diagonally dominant} if
    $C$ is $A$-bounded and $B$ is $D$-bounded;
  \item[(ii)\!] \hspace*{-0.5mm}\emph{diagonally $p$-dominant} if $C$ is $p$-subordinate
    to $A$ and $B$ is $p$-subordinate to~$D$.
\end{itemize}
\end{definition}

Note that for a diagonally dominant block operator matrix the domain of $T$ is given by the domains of the two diagonal entries,
$\mdef(T)=\mdef(A)\times \mdef(D)$. By Lemma~\ref{lem:psub} (ii), every diagonally $p$-dominant block operator matrix is diagonally dominant.

If we decompose a block operator matrix $T$ into its diagonal and off-diagonal~part,
\begin{equation}\label{eq:blockdecomp}
  T=S+R \quad\text{with}\quad
  S:=\pmat{A&0\\0&D},\quad R:=\pmat{0&B\\C&0},
\end{equation}
then $T$ is diagonally dominant if and only if $R$ is $S$-bounded, see \cite[\S2.2]{tretter-book}. A similar statement holds for diagonal $p$-dominance:

\begin{lemma}\label{lem:diagpdom}
  \begin{itemize}
  \item[{\rm (i)}]
   A block operator matrix $T$ is diagonally $p$-dominant if and only if
    $R$ is $p$-subordinate to $S$.
  \item[{\rm (ii)}]
    If\, $0\in\varrho(A)\cap\varrho(D)$,
    $C$ is $p_1$-subordinate to $A$, and $B$ is $p_2$-subordinate
    to $D$, then  $T$ is diagonally $p$-dominant with $p=\max\{p_1,p_2\}$.
  \end{itemize}
\end{lemma}

\begin{proof}
  (i) If $T$ is diagonally $p$-dominant, then H\"older's inequality yields that, for $x=(u,v)\in\mdef(T)=\mdef(S)$,
  \begin{align*}
    \|Rx\|^2&=\|Bv\|^2+\|Cu\|^2\leq c_B^2\|v\|^{2(1-p)}
    \|Dv\|^{2p}+c_C^2\|u\|^{2(1-p)}\|Au\|^{2p}\\
    &\leq \max\{c_B^2,c_C^2\}\left(\|u\|^2+\|v\|^2\right)^{1-p}\left(\|Au\|^2
    +\|Dv\|^2\right)^p\\
    &=\max\{c_B^2,c_C^2\}\|x\|^{2(1-p)}\|Sx\|^{2p}.
  \end{align*}
  Hence $R$ is $p$-subordinate to $S$.
  Vice versa, let $R$ be $p$-subordinate to $S$. Then 
  for $u\in\mdef(A)$ we have $x:=(u,0)\in\mdef(S)\subset\mdef(R)$,
  i.e.\ $u\in\mdef(C)$, and
  \[\|Cu\|=\|Rx\|\leq c\|x\|^{1-p}\|Sx\|^p=c\|u\|^{1-p}\|Au\|^p.\]
  An analogous argument yields that $B$ is $p$-subordinate to $D$.
  
  (ii) is an immediate consequence of Lemma~\ref{lem:psub}~(iii).
\end{proof}

\begin{definition}
A block operator matrix $T$ is called \emph{Hamiltonian} if $H_1=H_2=H$ and $T$ has the~form
\begin{equation}\label{eq:ham}
    T=\pmat{A&B\\C&-A^*} \quad \text{in} \ H \times H
  \end{equation}
with $A$ closed and $B$, $C$ symmetric; 
$T$ is called
\emph{nonnegative} if $B$, $C$ are nonnegative.
\end{definition}

\begin{remark}
  A Hamiltonian $T$ is diagonally dominant if and only if 
  $\mdef(A)\subset\mdef(C)$ and $\mdef(A^*)\subset\mdef(B)$,
  see \cite[Remark~2.2.2]{tretter-book}.
\end{remark}

\begin{lemma}\label{lem:nonnegham-gap}
  Let $T$ be a nonnegative diagonally dominant Hamiltonian operator matrix
  such that
  $\ri\R\subset\varrho(A)$.
  Then $\ri\R\subset r(T)$.
\end{lemma}

\begin{proof}
  Since $B,C$ are nonnegative symmetric, they admit nonnegative selfadjoint
  extensions. We may thus assume that $B$ and $C$ are selfadjoint.
  Then, for $t\in\R$, the operator $C^{1/2}(A-\ri t)^{-1}$ is defined
  on $H$ and closed; hence it is bounded by the closed graph theorem.
  Analogously, $B^{1/2}(A^*+\ri t)^{-1}$ is bounded.
  Suppose that $\ri t\not\in r(T)$. Then there exist
  $(u_n,v_n)\in\mdef(T)$ such that 
  \[
    \left\|\pmat{u_n\\v_n}\right\|=1, \ n\in \N, \qquad
    \lim_{n\to\infty}(T-\ri t)\pmat{u_n\\v_n}=
    \lim_{n\to\infty}\pmat{(A-\ri t)u_n+Bv_n\\Cu_n-(A^*+\ri t)v_n}=0.
  \]
  In view of $\|u_n\|\leq 1$, $\|v_n\|\le 1$, $n\in\N$, 
  the latter implies that
  \[\begin{aligned}
      &((A-\ri t)u_n|v_n)+(Bv_n|v_n)\ \to 0, \\
      &(Cu_n|u_n)-((A^*\!\!+\ri t)v_n|u_n)\to 0,
    \end{aligned}
    \qquad n\to\infty.\]
  Adding these two relations and taking the real part, we arrive at
  \[(Cu_n|u_n)+(Bv_n|v_n)\to 0, \qquad n\to\infty.\]
  Since $B,C$ are nonnegative, we obtain
  \[\|C^{1/2}u_n\|^2=(Cu_n|u_n)\to 0, \quad
    \|B^{1/2}v_n\|^2=(Bv_n|v_n)\to 0, \qquad n\to\infty.\]
  Because of $\|v_n\|\le 1$, $n\in\N$, the sequences $((A-\ri t)^{-1}v_n)_n$ and 
  $(C^{1/2}(A-\ri t)^{-1}v_n)_n$ are bounded and thus
  \begin{align*}
    0&=\lim_{n\to\infty}\bigl(Cu_n-(A^*+\ri t)v_n\big|(A-\ri t)^{-1}v_n\bigr)\\
    &=\lim_{n\to\infty}\Bigl(\bigl(C^{1/2}u_n\big|C^{1/2}(A-\ri t)^{-1}v_n\bigr)
      -\|v_n\|^2\Bigr)=-\lim_{n\to\infty}\|v_n\|^2,
  \end{align*}
  i.e.\ $v_n\to 0$, $ n\to \infty$. Analogous considerations for
  $((A-\ri t)u_n+Bv_n|(A^*+\ri t)^{-1}u_n)$ yield that $u_n\to 0$, $n\to\infty$,
  a contradiction to $\|(u_n,v_n)\|=1$, $n\in\N$.
\end{proof}

The following theorem is a perturbation result for Hamiltonians $T$ with sectorially dichotomous $A$; the corresponding spectral enclosure is displayed in Figure~\ref{fig:dichotpert}. 

\begin{theo}\label{theo:dichotham}
  Let $T$ be a nonnegative diagonally $p$-dominant Hamiltonian
  with $p<1$ and let $A$ be sectorially dichotomous with angle
  $\theta\in[0,\pi/2[$.
  Then $T$ is dichotomous, the spectral projections 
  $P_+$, $P_-$ satisfy
  \[\frac{1}{\pi \ri}\int_{\ri\R}^\prime(T-z)^{-1}x\,dz=P_+x-P_-x,
    \quad x\in H\times H,\]
  and there exist $h>0$, $\theta''\in\,]\theta,\pi/2[$ 
  and for every $\theta'\in\,]\theta,\pi/2[$ an $r'>0$ such that 
  \[\bigset{z\in\C}{|\Real z|<h}\cup\Omega_{\theta'',0}
    \cup\Omega_{\theta',r'}\subset\varrho(T).\]
\end{theo}

\begin{proof}
  We consider the decomposition $T=S+R$ into diagonal and off-diagonal~part, 
  \begin{equation}\label{eq:hamdecomp}
  T=S+R \quad\text{with}\quad
  S:=\pmat{A&0\\0&-A^*},\quad R:=\pmat{0&B\\C&0}.
\end{equation}
  Since $A$ and hence $A^*$ are sectorially dichotomous, see 
  Corollary~\ref{coroll:adjsectdichot},
  the same is true for $S$. Moreover $R$ is $p$-subordinate to $S$
  and $\ri\R\subset r(T)$.
  Thus Corollary~\ref{coroll:sectdichotpert} applies and yields all claims.
\end{proof}

A Hamiltonian $T$ as in \eqref{eq:ham} does not have any symmetry properties with respect to the scalar product in 
the Hilbert space $H\times H$. However, it exhibits some symmetries with respect to two different indefinite inner products on $H\times H$, see \cite[Section~5]{langer-ran-rotten}.

A vector space $V$ together with an indefinite inner product
$\braket{\cdot}{\cdot}$ is called a \emph{Krein space}
if there exists a  scalar product $(\cdot|\cdot)$ on $V$
and an involution $J:V\to V$ such that $(V,(\cdot|\cdot))$
is a Hilbert space and
\[\braket{x}{y}=(Jx|y), \quad
x,y\in V.\]
The so-called fundamental symmetry $J$ and the scalar product are not unique,
but the norms induced by two such scalar products are equivalent.

A subspace $U\subset V$ is called \emph{$J$-nonnegative} (\emph{$J$-nonpositive}, \emph{$J$-neutral}, respectively,)  
 if $\braket{x}{x}\geq 0$ ($\leq 0$, $=0$, respectively) for all $x\in U$.
A subspace $U$ is $J$-neutral if and only if 
it is contained in its $J$-orthogonal complement $U^\Jorth$, 
\[U \subset U^\Jorth :=\{x\in V\,|\,\braket{x}{y}=0 \text{ for all } y\in U\},\]
and it is called \emph{hypermaximal $J$-neutral} if $U=U^\Jorth$.

A linear operator $T$ on $V$ is called \emph{$J$-accretive} if 
$\Real\braket{Tx}{x}\geq 0$ for all $x\in\mdef(T)$.
A densely defined linear operator $T$ is called \emph{$J$-skew-symmetric} if
$\braket{Tx}{y}=-\braket{x}{Ty}$
for all $x,y\in\mdef(T)$. 
For more results on Krein spaces
and operators therein, we refer to \cite{azizov-iokhvidov, krein65}.

\begin{prop}
  Let $V$ be a Krein space with fundamental symmetry $J$ and let $\,T$ be a dichotomous operator on $V$
  with corresponding decomposition $V=V_+\oplus V_-$ and 
  projections $P_+$, $P_-$ such that
  \[\frac{1}{\pi \ri}\int_{\ri\R}^\prime(T-z)^{-1}x\,dz=P_+x-P_-x,
    \quad x\in V.\]
  \begin{itemize}
  \item[{\rm (i)}]
  If\, $T$ is $J$-accretive, then $V_+$ is $J$-nonnegative and $V_-$ is
  $J$-nonpositive.
  \item[{\rm (ii)}]
  If\, $T$ is $J$-skew-symmetric, then $V_+$ and $\,V_-$ are hypermaximal
  $J$-neutral.
  \end{itemize}
\end{prop}

\begin{proof}
  (i) The simple proof was given in  \cite[Theorem~1.4]{langer-tretter}; 
  e.g.\ for $x\in V_+$ it is nothing but the inequality
  \begin{align*}
    \braket{x}{x}&=\Real\,\braket{P_+x-P_-x}{x}
    =\frac{1}{\pi}\int_\R^\prime \Real\,\braket{(T-\ri t)^{-1}x}{x}\,dt\\
    &=\frac{1}{\pi}\int_\R^\prime
    \Real\,\braket{T(T-\ri t)^{-1}x}{(T-\ri t)^{-1}x}\,dt
    \geq 0.
  \end{align*}

  (ii) If $T$ is $J$-skew-symmetric, then both $T$ and $-T$ are $J$-accretive.
  Consequently, $V_\pm$ are both nonnegative and nonpositive, thus neutral.
  To prove hypermaximal neutrality, let e.g.\ $x\in V_+^\Jorth$.
  Using the decomposition $V=V_+\oplus V_-$,
  we write $x=u+v$ with $u\in V_+$, $v\in V_-$.
  If $v\neq0$, then there exists $y\in V$ such that
  $\braket{v}{y}\neq0$ (e.g.\ one may choose $y=Jv$). Since $V_-$ is neutral,
  we may assume that $y\in V_+$. The neutrality of $V_+$ then implies
  that $\braket{x}{y}=\braket{v}{y}\neq0$, in contradiction to
  $x\in V_+^\Jorth$. Therefore $v=0$, i.e.\ $x\in V_+$.
\end{proof}

Following \cite[Section~5]{langer-ran-rotten}, we equip the product space $H\times H$ with two different indefinite inner products, given by
$\braket{x}{y}_1:=(J_1x|y)$ and $[x|y]_2:=(J_2x|y)$
with the fundamental symmetries
\begin{equation}
\label{fundsymm}
  J_1:=\pmat{0&-\ri I\\\ri I&0},\qquad J_2:=\pmat{0&I\\I&0}.
\end{equation}
As in the case of bounded $B$ and $C$, the Hamiltonian has the following symmetry properties with respect to 
$J_1$ and $J_2$.

\begin{lemma}
  The  Hamiltonian operator matrix $T$ is $J_1$-skew-symmetric, and
  $T$ is nonnegative if and only if it is $J_2$-accretive.
\end{lemma}

\begin{proof}
  The  assertions are immediate from
  \begin{align*}
  [T(u,v)|(u,v)]_1&=\ri\,(2\Real(Au|v)+(Bv|v)-(Cu|u))\in\ri\R,\\
  \Real\,[T(u,v)|(u,v)]_2&=(Bv|v)+(Cu|u). \\[-13mm]
  \end{align*}
  \end{proof}
  

\begin{coroll}\label{coroll:hypmax}
  In the situation of Theorem~\emph{\ref{theo:dichotham}}, let 
  $H\times H=V_+\oplus V_-$ be the decomposition corresponding to the 
  dichotomy of\, $T$. 
  Then $V_+,V_-$ are hypermaximal $J_1$-neutral, $V_+$ is $J_2$-nonnegative,
  and $V_-$ is $J_2$-nonpositive.
\end{coroll}


\section{Invariant graph subspaces and Riccati equations}
\label{sec:ricc}

There is a close relation between the invariance of \emph{graph~subspaces} 
\[
  \graph(X)=\Bigl\{\pmat{u\\Xu}\,\Big|\,u\in\mdef(X)\Bigr\}
\]
of linear operators $X$ on a Hilbert space $H$ under a block operator matrix  and
solutions of Riccati equations, see e.g.\ 
\cite{bart-gohberg-kaashoek79,kostrykin-makarov-motovilov,wyss-rinvsubham};
in our setting it reads as follows:

\begin{lemma}\label{lem:riccsol}
  Let $T$ be a diagonally dominant Hamiltonian and $X$ a linear operator on $H$.
  Then the graph subspace $\graph(X)$ is $T$-invariant
  if and only if $X$ is
  a solution of the Riccati equation
  \begin{equation}\label{eq:are}
    (A^*X+X(A+BX)-C)u=0, \qquad
    u\in\mdef(A)\cap X^{-1}\mdef(A^*);
  \end{equation}
  in particular, $(A+BX)u\in\mdef(X)$ for all
  $u\in\mdef(A)\cap X^{-1}\mdef(A^*)$.
\end{lemma}
\begin{proof}
  $\graph(X)$ is $T$-invariant if and only if for
  all $u\in\mdef(A)\cap\mdef(X)$ with $Xu\in\mdef(A^*)$ there exists
  $v\in\mdef(X)$ such that 
  \[\pmat{Au+BXu\\Cu-A^*Xu}=T\pmat{u\\Xu}=\pmat{v\\Xv},\]
  and this is obviously equivalent to \eqref{eq:are}.
\end{proof}

Graph subspaces are closely related to the Krein space fundamental symmetries
$J_1,J_2$ introduced in \eqref{fundsymm}, see also
\cite{dijksma-desnoo}:

\begin{lemma}[\mbox{\cite[Lemma~6.2]{wyss-rinvsubham}}] 
  \label{lem:graphJsym}
  Let $X$ be a linear operator on $H$. Then
  \begin{itemize}
  \item[{\rm (i)}] $X$ is selfadjoint if and only if $\,\graph(X)$ is hypermaximal $J_1$-neutral;
  \end{itemize}
  if $\,X$ is symmetric, then
  \begin{itemize}    
  \item[{\rm (ii)}] $X$
    is nonnegative $($nonpositive, respectively$)$ if and only if $\,\graph(X)$ is $J_2$-non\-negative $(J_2$-nonpositive, respectively$)$.
  \end{itemize}
\end{lemma}

The next theorem generalises \cite[Theorem~5.1]{langer-ran-rotten}
where the off-diagonal operators $B$ and $C$ were assumed to be bounded,
and it complements results in
\cite{kuiper-zwart,wyss-rinvsubham,wyss-unbctrlham} 
where Hamiltonians $T$ possessing a Riesz basis of generalised
eigenvectors but without dichotomy were investigated.

\begin{theo}\label{theo:ham-ricc}
  Let $T$ be a nonnegative diagonally $p$-dominant Hamiltonian
  with $p<1$ such that $A$ is sectorially dichotomous and
  \begin{equation}\label{eq:approxcontr}
    \bigcap_{t\in\R}\ker (B(A^*+\ri t)^{-1})=\{0\}.
  \vspace{-2mm}  
  \end{equation}
  Then 
  \begin{itemize}
  \item[{\rm (i)}]
  $T$ is dichotomous and its spectral subspaces
  are graph subspaces, $V_\pm\!\!=\!\graph(X_\pm)$;
  \item[{\rm (ii)}]
  $X_\pm$ are selfadjoint, $X_+$ is nonnegative, and $X_-$ is nonpositive;
  \item[{\rm (iii)}]
  $\mdef(A)\cap X_\pm^{-1}\mdef(A^*)$ are a core for $X_\pm$
  and $X_\pm$ satisfy the Riccati equations
  \begin{equation}\label{eq:riccati}
    (A^*X_\pm+X_\pm(A+BX_\pm)-C)u=0, \qquad
    u\in\mdef(A)\cap X_\pm^{-1}\mdef(A^*).
  \end{equation}
  \end{itemize}
\end{theo}

\begin{proof}
  (i) By Theorem~\ref{theo:dichotham} and Corollary~\ref{coroll:hypmax},
  $T$ is dichotomous, $V_+$, $V_-$ are hypermaximal $J_1$-neutral,
  $V_+$ is $J_2$-nonnegative, and $V_-$ is $J_2$-non\-po\-si\-tive.
  To show that $V_\pm=\graph(X_\pm)$
  with some linear operator $X_\pm$, 
  it suffices to show that
  $(0,w)\in V_\pm$ implies $w=0$.
  Setting $(u,v):=(T-\ri t)^{-1}(0,w)$ for $t\in \R$, we have
  \[(A-\ri t)u+Bv=0,\quad Cu-(A^*+\ri t)v=w.\]
  Since $V_\pm$ is $J_1$-neutral and invariant under $(T-\ri t)^{-1}$,
  this implies that
  \[0=\Braket{\pmat{0\\w}}{\pmat{u\\v}}_1=-\ri(w|u)\]
  and thus
  \[0=(w|u)=(Cu|u)-(v|(A-\ri t)u)=(Cu|u)+(Bv|v).\]
  Since $B$ and $C$ are nonnegative, it follows that $0=(Cu|u)=(Bv|v)$.
  Thus, for all $r\in\R$ and $\tilde{v}\in \mdef(B)$,
  \[0\leq (B(rv+\tilde{v})|rv+\tilde{v})=2r\Real(Bv|\tilde{v})
    +(B\tilde{v}|\tilde{v}),\]
  which yields $Bv=0$.
  Similarly, we obtain  $Cu=0$ and so $w=-(A^*+\ri t)v$.
  We conclude that
  \(B(A^*+\ri t)^{-1}w=-Bv=0\).
  As $t\in\R$ was arbitrary,
  \eqref{eq:approxcontr} implies that $w=0$.

  (ii), (iii) Since $V_\pm=\graph(X_\pm)$ are hypermaximal $J_1$-neutral
  and $J_2$-nonnegative/-non\-po\-si\-tive, 
  Lemma \ref{lem:graphJsym} shows that
  $X_\pm$ are selfadjoint and nonnegative/non\-po\-si\-tive, 
  while Lemma \ref{lem:riccsol} shows that $X_\pm$ satisfy
  \eqref{eq:riccati}.
  Moreover, we have $(u,X_\pm u)\in\mdef(T)$ if and only if 
  $u\in\mdef(A)\cap X_\pm^{-1}\mdef(A^*)$.
  Since $V_\pm\cap\mdef(T)$ are dense in $V_\pm$, this implies that
  $\mdef(A)\cap X_\pm^{-1}\mdef(A^*)$ are a core for $X_\pm$.
\end{proof}

Next we derive necessary as well as sufficient conditions for assumption~\eqref{eq:approxcontr}. 

\begin{prop}\label{prop:apctrlcond}
  Let $A$ be a closed densely defined linear operator with $\ri\R\subset\varrho(A)$ and $B$ symmetric
  with $\mdef(A^*)\subset\mdef(B)$.
  Then the assertions
  \begin{itemize}
  \item[{\rm (i)}] $\ker B=\{0\}$,
  \item[{\rm (ii)}] $\bigcap_{t\in\R}\ker (B(A^*+\ri t)^{-1})=\{0\}$,  \ \ see \eqref{eq:approxcontr},
  \item[{\rm (iii)}]
    $\overline{\mspan\set{(A-\ri t)^{-1}B^*u}{t\in\R,\,u\in\mdef(B^*)}}= H$,
  \item[{\rm (iv)}] $\forall \:\lambda\in\sigma_p(A^*): \ 
    \ker B\cap\ker(A^*-\lambda)=\{0\}$
  \end{itemize}
  satisfy the \vspace{-1.5mm} implications 
   $${\rm (i)}\Longrightarrow{\rm (ii)}\Longleftrightarrow{\rm (iii)}\Longrightarrow{\rm (iv)};$$
%
  if $A$ has compact resolvent and possesses a complete system
  of generalised eigenvectors, then \vspace{-1.5mm} even 
  $${\rm (iii)}\Longleftrightarrow{\rm (iv)}.$$
\end{prop}

\begin{remark}
  For the special case of bounded $B$, instead of
  \eqref{eq:approxcontr}, the equivalent condition (iii) in
  Proposition~\ref{prop:apctrlcond} was used in
  \cite[Theorem~5.1]{langer-ran-rotten}.  For the special case of
  normal $A$ with compact resolvent, the equivalence
  (iii)$\Leftrightarrow$(iv) in Proposition~\ref{prop:apctrlcond} was
  established in \cite[Proposition~6.6]{wyss-rinvsubham}.
\end{remark}

For the proof of Proposition \ref{prop:apctrlcond} we need the following lemma.

\begin{lemma}\label{lem:nonctrlsubsp}
  Let $A$ be a closed densely defined linear operator with
  $\ri\R\subset\varrho(A)$ and $B$ symmetric
  with $\mdef(A^*)\subset\mdef(B)$.
  Let $\rho_0$ be the connected component of $\varrho(A)$ containing~$\ri\R$.
  If $\rho\subset \rho_0$ has an accumulation point in $\rho_0$, then
  \begin{align*}
    \bigcap_{t\in\R}\ker (B(A^*+\ri t)^{-1})
    &=\bigcap_{z\in \rho}\ker (B(A^*-\bar{z})^{-1})\\[-1mm]
    &=\Bigl(\mspan\bigset{(A-z)^{-1}B^*u}{z\in \rho,\,u\in\mdef(B^*)}
    \Bigr)^\perp.
  \end{align*}
\end{lemma}

\begin{proof}
  The second identity is immediate from the identities
  \[
  \ker (B(A^*-\bar z)^{-1}) = \range(( B(A^*-\bar z)^{-1} )^*)^\perp
  = \range((A-z)^{-1}B^*)^\perp;
  \]
  note that
  we have used that $B(A^*-\bar z)^{-1}$ is bounded
  and that $(B(A^*-\bar z)^{-1})^*|_{\mdef(B^*)}=(A-z)^{-1}B^*$.
  Moreover, by the identity theorem, if $((A-z)^{-1}B^*u|x)=0$ for all 
  $z\in \rho$, then this continues to hold for all $z\in \rho_0$ and thus 
  \[\bigcap_{z\in \rho}\ker (B(A^*-\bar{z})^{-1})
  =\bigcap_{z\in \rho_0}\ker (B(A^*-\bar{z})^{-1}).\]
  Since $\ri\R$ is one possible choice for $\rho$, the proof is complete.
\end{proof}

\begin{proof}[Proof \emph{(of Proposition \ref{prop:apctrlcond})}]
  The implication (i)$\Rightarrow$(ii) is clear and (ii)$\Leftrightarrow$(iii) follows from
  Lemma~\ref{lem:nonctrlsubsp}. For the implication (ii)$\Rightarrow$(iv) we observe that if $\lambda\in \sigma_p(A^*)$ and
  $x\in\ker B\cap\ker(A^*-\lambda)$, then
  $B(A^*+\ri t)^{-1}x=(\lambda+\ri t)^{-1}Bx=0$ for all $t\in\R$ and hence $x=0$ by (ii).

  To show the reverse implication (iv)$\Rightarrow$(ii) under the additional 
  assumptions on~$A$,
  we first prove that the closed subspace
  \[\cN:=\bigcap_{t\in\R}\ker (B(A^*+\ri t)^{-1})\]
  is $(A^*-z)^{-1}$-invariant for every $z\in\varrho(A^*)$.
  Let $x\in\cN$. 
  Since $A$ has compact resolvent, $\varrho(A)$ is connected.
  Thus Lemma~\ref{lem:nonctrlsubsp} implies that  $B(A^*-z)^{-1}x=0$ for all 
  $z\in\varrho(A^*)$. 
  Hence, by the resolvent identity, we find that for all $t\in\R$, 
  $z\neq-\ri t$,
  \[B(A^*+\ri t)^{-1}(A^*-z)^{-1}x=\frac{1}{\ri t+z}\bigl(
  B(A^*-z)^{-1}x-B(A^*+\ri t)^{-1}x\bigr)=0.\]
  Therefore $\cN$ is $(A^*-z)^{-1}$-invariant for all 
  $z\in\varrho(A^*)\setminus\ri\R$ and thus,
  by continuity, for all $z\in\varrho(A^*)$.

  Secondly, we use induction on $n\in\N$ to show that
  that $\cN\cap\ker(A^*-\lambda)^n=\{0\}$ for~all 
  $\lambda\!\in\!\sigma_p(A^*)$.
  The case $n=0$ is trivial. For $n\geq1$ let
  $x\!\in\!\cN\cap\ker(A^*\!-\lambda)^n$ and set $y:=(A^*-\lambda)x$.
  Since $A$ was assumed to have compact resolvent, the subspace
  $\cN\cap\ker(A^*-\lambda)^n$ has  finite dimension; 
  by the first part of the proof it is invariant
  under $(A^*-z)^{-1}$ and hence also under $A^*$. 
  Therefore $y\in\cN\cap\ker(A^*-\lambda)^{n-1}$.
  By induction this yields $y=0$.
  Hence $x\in\ker(A^*-\lambda)$ and
  $0=B(A^*+\ri t)^{-1}x=(\lambda+\ri t)^{-1}Bx$; thus $Bx=0$.
  From (iv) we then obtain $x=0$. 

  Now let $\lambda\in\sigma_p(A)$ be arbitrary and let $P$ be the Riesz 
  projection onto the  corresponding generalised eigenspace of $A$.
  Then the Riesz projection corresponding to the eigenvalue 
  $\bar\lambda$ of $A^*$ is given by
  \[ P^*=\frac{\ri}{2\pi}\int_{\partial B_\eps(\bar\lambda)}(A^*-z)^{-1}\,dz\]
  with $\varepsilon>0$ such that $\overline{B_\eps(\bar\lambda)}\setminus \{\bar\lambda\} \subset \varrho(A^*)$.
  Since $\cN$ is $(A^*-z)^{-1}$-invariant and closed, it is also invariant
  under $P^*$. Moreover, $\range(P^*)=\ker(A^*-\bar\lambda)^n$ for some $n\in\N$.
  For $x\in\cN$ we obtain $P^*x\in\cN\cap\ker(A^*-\bar\lambda)^n$ and so
  $P^*x=0$, i.e.\ $x\perp\range(P)$.
  Since $\lambda\in\sigma_p(A)$ was arbitrary, $x$ is orthogonal to the system 
  of gen\-eralised eigenvectors of $A$, which was assumed to be complete, 
  hence $x=0$. 
\end{proof}

\begin{remark}
  If, in addition to being sectorially dichotomous, $A$ generates a strong\-ly
  continuous semigroup and $B$ is bounded, then
  \eqref{eq:approxcontr} is equivalent to the 
  approximate controllability of the pair $(A,B)$,
  compare \cite[\S4.1]{curtain-zwart}.
\end{remark}

\begin{remark}
  There is a second Riccati equation corresponding to the Hamiltonian~$T$:
  A linear operator $Y$ in the Hilbert space $H$ is a solution of 
  the Riccati equation
  \begin{equation}\label{eq:dual-are}
    (AY+Y(A^*-CY)+B)v=0, \qquad
    v\in\mdef(A^*)\cap Y^{-1}\mdef(A),
  \end{equation}
  if and only if the ``inverse'' graph subspace
  \[\invgraph(Y):=\biggl\{\pmat{Yv\\v}\,\Big|\,v\in\mdef(Y)\biggr\}\]
  is $T$-invariant.
  The Riccati equations \eqref{eq:are} and \eqref{eq:dual-are} are 
  dual to each other in the following sense:
  $\invgraph(Y)$ is $T$-invariant if and only if $\graph(Y)$ is invariant
  under the transformed Hamiltonian
  \[\widetilde{T}
    =\pmat{0&I\\I&0}\pmat{A&B\\C&-A^*}\pmat{0&I\\I&0}
    =\pmat{-A^*&C\\B&A}.\]
  For example, the dual version of Theorem~\ref{theo:ham-ricc} states
  that if, instead of \eqref{eq:approxcontr}, 
  \begin{equation*}
    \bigcap_{t\in\R}\ker C(A-\ri t)^{-1}=\{0\}
  \end{equation*}
  holds, then $V_\pm=\invgraph(Y_\pm)$ where $Y_\pm$ is a selfadjoint
  nonnegative/nonpositive solution of \eqref{eq:dual-are},
  and $\mdef(A^*)\cap Y_\pm^{-1}\mdef(A)$ is a core for $Y_\pm$.
\end{remark}

\section{Bounded solutions of Riccati equations}
\label{sec:bndsol}

In this section we consider Hamiltonians $T$ for which $A$ is a sectorial
operator with angle $\theta<\pi/2$. Then the spectra of the diagonal
entries $A$ and $-A^*$ of $T$ lie in the sectors $\Sigma_\theta$ and
$\Sigma_{-\theta}$ in the right and left half-plane, respectively.

We show that then the solution $X_+$ of the Riccati equation in
Theorem~\ref{theo:ham-ricc} is bounded and uniquely determined; 
if $-A$ is sectorial, then $X_-$ is bounded and uniquely determined.

\begin{lemma}\label{lem:bndsol1}
  Let $T$ be a nonnegative diagonally $p$-dominant Hamiltonian 
  with \linebreak 
  $p<1$ and let $A$ be sectorially dichotomous.
  If the linear operator $X:H\to H$ is bounded and  $\graph(X)$ is
  invariant under $T$ and under
  $(T-z)^{-1}$, $z\in\varrho(T)$,
  then $X\mdef(A)\subset\mdef(A^*)$ and
  $X$ is a solution of the Riccati equation
  \begin{equation}\label{eq:bndsol}
    (A^*X+XA+XBX-C)u=0, \quad u\in\mdef(A).
  \end{equation}
\end{lemma}

\begin{proof}
  We consider the isomorphism $\varphi$ and the projection $\proj_1$
  given by
  \begin{equation}\label{last}
  \begin{aligned}
      \varphi:H&\to\graph(X),\\ u&\mapsto(u,Xu),
    \end{aligned}\qquad
    \begin{aligned}
      \proj_1:H\times H&\to H,\\ (u,v)&\mapsto u,
    \end{aligned}
  \end{equation}
  which are related by $\varphi^{-1}=\proj_1|_{\graph(X)}$.
  Using the decomposition $T=S+R$ from \eqref{eq:hamdecomp} into diagonal and off-diagonal part,
  we define the operators $E:=\proj_1T\varphi$ and $F:=\proj_1R\varphi$
  on $H$, i.e.\
  \begin{alignat*}{2}
    &\mdef(E)=\mdef(A)\cap X^{-1}\mdef(A^*),\qquad
    &&Eu=Au+BXu,\\
    &\mdef(F)=\mdef(C)\cap X^{-1}\mdef(B),
    &&Fu=BXu.
  \end{alignat*}
  By assumption, $\mdef(A) \subset \mdef(C)$, 
  $\mdef(A^*) \subset \mdef(B)$ so that $\mdef(E-F)=\mdef(E)\subset\mdef(A)$
  and $(E-F)u=Au$; hence $E-F$ is a restriction of $A$.
  We aim to show that, in fact, $\mdef(E-F)=\mdef(A)$.

  Since $\graph(X)$ is $T$-invariant,
  $E=\varphi^{-1}T|_{\graph(X)}\varphi$ and hence
  $\varrho(E)=\varrho(T|_{\graph(X)})$.  Furthermore, we have
  $\varrho(T)\subset\varrho(T|_{\graph(X)})$ since $\graph(X)$ is also
  $(T-z)^{-1}$-invariant.  By Theorem~\ref{theo:dichotham} the
  operator $T$ is dichotomous and thus $\ri\R\subset\varrho(T)
  \subset \varrho(E)$.  
  From $\|R(S-\ri t)^{-1}\|\leq M/|t|^{1-p}$
  with some $M>0$, see \eqref{eq:psubest} and Lemma~\ref{lem:diagpdom},
  and from
  \begin{align*}
    F(E-\ri t)^{-1}&=\proj_1R\varphi\varphi^{-1}(T-\ri t)^{-1}\varphi
    =\proj_1R(T-\ri t)^{-1}\varphi\\
    &=\proj_1R(S-\ri t)^{-1}\bigl(I+R(S-\ri t)^{-1}\bigr)^{-1}\varphi
  \end{align*}
  we see that $\|F(E-\ri t)^{-1}\|<1$ for large $|t|$. 
  Consequently, $\ri t\in\varrho(E-F)$ for large~$|t|$.
  Since $\ri t\in\varrho(A)$ for all $t\in\R$ and $E-F$ is a restriction 
  of $A$, this implies
  that $\mdef(A)=\mdef(E-F)=\mdef(A)\cap X^{-1}\mdef(A^*)$, i.e.\
  $X\mdef(A)\subset\mdef(A^*)$. The Riccati equation \eqref{eq:bndsol}
  now follows from Lemma~\ref{lem:riccsol}.
\end{proof}

\begin{lemma}\label{lem:bndsol2}
  Let $T$ be a nonnegative diagonally $p$-dominant Hamiltonian
  with \linebreak $p<1$.
  Let $A$ be sectorial with angle $\theta<\pi/2$ and $0\in\varrho(A)$.
  If\, $X$ is a bounded selfadjoint
  solution of \eqref{eq:bndsol} with $X\mdef(A)\subset\mdef(A^*)$,
  then there exists a constant $L=L(A,p,c_C)$ 
  depending only on $A$, $p$, and the $p$-subordination bound $c_C$ of\, $C$ to $A$ such that
  \[(Xu|u)\leq L\|u\|^2, \quad u\in H;\]
  in particular, $\|X\|\leq L$ if $X$ is nonnegative.
\end{lemma}

\begin{proof}
  From \eqref{eq:bndsol} and since $T$, and thus $B$, is nonnegative, we obtain
  \[(Au|Xu)+(Xu|Au)=(Cu|u)-(BXu|Xu)\leq (Cu|u), \qquad u\in\mdef(A).\]
  Hence, for arbitrary $t\in\R$,
  \[2\Real\bigl((A-\ri t)u\big|Xu\bigr)\leq(Cu|u), \qquad u\in\mdef(A).\]
  Together with the $p$-subordinacy of $C$ to $A$, this implies that for arbitrary $v\in H$, letting $u:=(A-\ri t)^{-1}v$, 
  \begin{align*}
    2\Real\bigl(v\big|X(A-\ri t)^{-1}v\bigr)
    &\leq\bigl(C(A-\ri t)^{-1}v\big|(A-\ri t)^{-1}v\bigr)\\
    &\leq\|C(A-\ri t)^{-1}\|\,\|(A-\ri t)^{-1}\|\|v\|^2\\
    &\leq c_C \|(A-\ri t)^{-1}\|^{2-p} \|A(A-\ri t)^{-1}\|^p\|v\|^2.
  \end{align*}
  Lemma~\ref{lem:sectdichot} applied to the sectorial operator $A$ (for which $P_-=0$) yields that
  \begin{align*} 
  & \frac{1}{\pi}\int_{\R}'(A-\ri t)^{-1}v\,dt=v,\qquad v\in V,
  \\
  &\|(A-\ri t)^{-1}\|\leq\frac{M}{|t|},\qquad t\in\R\setminus\{0\};
  \end{align*}
  in particular,
  $\|A(A-\ri t)^{-1}\|$ is uniformly bounded in $t\in\R$. Altogether, we obtain
  \begin{align*}
    2\pi(Xv|v)&= 2\pi \Real (Xv|v)
    =\int_\R'2\Real\bigl(X(A-\ri t)^{-1}v\big|v\bigr)\,dt  \\
    &\leq c_C  \Big( \int_\R\|(A-\ri t)^{-1}\|^{2-p}\,dt\Big) \, \sup_{t\in\R}\|A(A-\ri t)^{-1}\|^p
    \|v\|^2 
    \\
    &=: L(A,p,c_C) \|v\|^2. \\[-13mm]
  \end{align*}
\end{proof}

\vspace{0.5mm}

The following proposition is the crucial step in proving the boundedness of a
solution of the Riccati equation \eqref{eq:riccati} in the presence of
unbounded $B$ and $C$.

\begin{prop}\label{prop:contgraph}
  For $\,r\in[0,1]$, let $X_r$ be linear operators on $H$ and $P_r$ projections 
  on $H\times H$ such that $\range(P_r)=\graph(X_r)$.
  Suppose that
  \begin{enumerate}
  \item[{\rm (i)}] $P_r$ depends continuously on $r$ in the uniform
    operator topology;
  \item[{\rm (ii)}] $X_0$ is bounded;
  \item[{\rm (iii)}] there exists $L>0$ so that for all $r\in[0,1]$,
    if $X_r$ is bounded, then $\|X_r\|\leq L$.
  \end{enumerate}
  Then all $X_r$, $r\in[0,1]$, are bounded.
\end{prop}

\begin{proof}
  Let $J:=\set{r\in[0,1]}{X_r\text{ bounded}}$. Then, by assumption (ii),
  $0\in J$.  We will show that $J$ is closed and open in the interval $[0,1]$
  and hence equal to $[0,1]$.

  Let $(r_n)_{n\in\N}\subset J$, $\lim_{n\to\infty}r_n=r$, and
  $u\in\mdef(X_r)$. Set
  \[x:=\pmat{u\\X_ru}, \quad P_{r_n}x=:\pmat{u_n\\X_{r_n}u_n}, \qquad n\in\N.\]
  Then $\lim_{n\to\infty}P_{r_n}x= P_rx=x$, which implies that
  $u_n\to u$ and $X_{r_n}u_n\to X_ru$ as $n\to\infty$. 
  By assumption (iii), we obtain
  \[\|X_ru\|=\lim_{n\to\infty}\|X_{r_n}u_n\|\leq L\lim_{n\to\infty}\|u_n\|
    =L\|u\|\]
  and hence $r\in J$. Therefore, $J$ is closed.

  Now suppose that $J$ is not open. Then there exist
  $r\in J$ and
  $(r_n)_{n\in\N}\subset[0,1]\setminus J$ such that  
  $\lim_{n\to\infty}r_n=r$.
  So all $X_{r_n}$ are unbounded. Hence there are $u_n\in\mdef(X_{r_n})$ 
  with $\|u_n\|\leq1/n$ and
  $\|X_{r_n}u_n\|=1$. Set
  \[x_n:=\pmat{u_n\\X_{r_n}u_n}, \quad P_rx_n=:\pmat{w_n\\X_rw_n}, \qquad n\in\N.\]
  Since $x_n\in \graph(X_{r_n})=\range(P_{r_n})$ for all $n\in\N$,
  we have $P_{r_n}x_n=x_n$ and
  \begin{align*}
    1&=\|X_{r_n}u_n\|
    \leq\|X_{r_n}u_n-X_rw_n\|+\|X_r\|\bigl(\|w_n-u_n\|+\|u_n\|\bigr)\\
    &\leq\|P_{r_n}x_n-P_rx_n\|+\|X_r\|\bigl(\|P_rx_n-P_{r_n}x_n\|
      +\|u_n\|\bigr)\\
    &\leq(1+\|X_r\|)\|P_{r_n}-P_r\|\,\|x_n\|+\|X_r\|\,\|u_n\|.
  \end{align*}
  Since
  $\|x_n\|^2=\|u_n\|^2+\|X_{r_n}u_n\|^2\leq 1/n^2+1$
  and $P_{r_n}\to P_r$, $u_n\to0$ as 
  $n\to\infty$, this is a contradiction.
\end{proof}

\begin{theo}\label{theo:bndricc}
  Let $T$ be a nonnegative diagonally $p$-dominant Hamiltonian with
  $p<1$. Suppose that $\ri\R\subset\varrho(A)$ and assumption \eqref{eq:approxcontr} holds, i.e.\
  \[
    \bigcap_{t\in\R}\ker (B(A^*+\ri t)^{-1})=\{0\}.\]
  \begin{itemize}
  \item[{\rm (i)}]
    If $A$ is sectorial with angle  $\theta<\pi/2$,
    then the nonnegative solution $X_+$ of the Riccati equation \eqref{eq:riccati}
    in Theorem~\emph{\ref{theo:ham-ricc}} is bounded,
    satisfies $X_+\mdef(A)\subset\mdef(A^*)$ and hence
    \begin{equation}\label{eq:bndricc}
      (A^*X_++X_+A+X_+BX_+-C)u=0, \quad u\in\mdef(A).
    \end{equation}
    Moreover, $X_+$ is the uniquely determined
    bounded nonnegative operator such that
    $X_+\mdef(A)\subset\mdef(A^*)$ and \eqref{eq:bndricc} hold.
  \item[{\rm (ii)}]
    If $-A$ is sectorial with angle $\theta<\pi/2$,
    then the nonpositive solution $X_-$ 
    of the Riccati equation \eqref{eq:riccati}
    in Theorem~\emph{\ref{theo:ham-ricc}} is bounded,
    satisfies $X_-\mdef(A)\subset\mdef(A^*)$, and  hence 
    \begin{equation}\label{eq:bndricc-neg}
    (A^*X_-+X_-A+X_-BX_--C)u=0, \quad u\in\mdef(A).
    \end{equation}
    Moreover, $X_-$ is the uniquely determined
    bounded nonpositive operator such that
    $X_-\mdef(A)\subset\mdef(A^*)$ and \eqref{eq:bndricc-neg} hold.
  \end{itemize}
\end{theo}

\begin{proof}
  Suppose that $A$ is sectorial.
  Consider the family of
  operators $T_r=S+rR$, $r\in[0,1]$, where $S$, $R$ are 
  the diagonal and off-diagonal part of $T$, respectively,
  as in~\eqref{eq:hamdecomp}.
  By Theorem~\ref{theo:dichotham}, each $T_r$ is dichotomous and
  the corresponding projections $P_{r,+}$ and $P_{r,-}$ satisfy
  \[2P_{r,+} x-x= P_{r,+}x - P_{r,-}x 
  =\frac{1}{\pi}\int_{\R}'(T_r-\ri t)^{-1}x\,dt,
  \qquad x\in H\times H.\]
  For $r>0$, Theo\-rem~\ref{theo:ham-ricc} applies to $T_r$
  since $\ker (B(A^*+\ri t)^{-1}) =\ker (rB(A^*+\ri t)^{-1})$ if $r>0$.
  Hence there are nonnegative selfadjoint operators $X_r$, $r>0$, such that 
  $\range(P_{r,+})=\graph(X_r)$;
  in particular, $X_1=X_+$.
  For $r=0$ we have $T_0=S$ and
  $\range(P_{0,+})=H\times \{0\}=\graph(X_0)$ where $X_0=0$,
  see also Lemma~\ref{lem:sectdichot}.
  If we set $P_r := P_{r,+}$, then we obtain, for $r,s\in[0,1]$,
  \begin{equation}
  \label{tuebingen2}
  P_rx-P_sx=\frac{1}{2\pi}\int_\R'\bigl((S+rR-\ri t)^{-1}
    -(S+sR-\ri t)^{-1}\bigr)x\,dt, \quad x\in H\times H.\end{equation}
  Since
  $\|R(S-\ri t)^{-1}\|\leq M'/|t|^{1-p}$, see \eqref{eq:psubest}, and
  $r\in[0,1]$, a Neumann series argument yields
  $\|(I+rR(S-\ri t)^{-1})^{-1}\|\leq2$ for $|t|\geq t_0$, where
  the constant $t_0>0$ is independent of $r$.
  Using  $\|(S-\ri t)^{-1}\|\leq M/|t|$ and
  \begin{align*}
    &(S+rR-\ri t)^{-1}-(S+sR-\ri t)^{-1}\\
    &=(S+sR-\ri t)^{-1}(s-r)R(S+rR-\ri t)^{-1}\\
    &=(S-\ri t)^{-1}\bigl(I+sR(S-\ri t)^{-1}\bigr)^{-1}(s-r)R(S-\ri t)^{-1}
      \bigl(I+rR(S-\ri t)^{-1}\bigr)^{-1},
  \end{align*}
  we find that
  \begin{equation}\label{eq:bndricc-r1}
    \|(S+rR-\ri t)^{-1}-(S+sR-\ri t)^{-1}\|\leq|s-r|\cdot\frac{4MM'}{|t|^{2-p}}
    \quad\text{for}\quad |t|\geq t_0.
  \end{equation}
  The identity
  \[(S+sR-\ri\tilde t)^{-1}=(S+rR-\ri t)^{-1}
  \bigl(I+((s-r)R+\ri t-\ri\tilde t)(S+rR-\ri t)^{-1}\bigr)^{-1}\] 
  implies that 
  the mapping $(r,t)\mapsto(S+rR-\ri t)^{-1}$ is continuous in the
  operator norm topology.
  On the compact set
  $\set{(r,t)\in \R^2}{r\in[0,1],\,|t|\leq t_0}$
  it is thus uniformly continuous.
  Hence for $\eps>0$ there exists $\delta>0$ such that
  \begin{equation}\label{eq:bndricc-r2}
    \|(S+rR-\ri t)^{-1}-(S+sR-\ri t)^{-1}\|<\eps \quad\text{for}\quad
    |t| \leq t_0,\,
    |s-r|<\delta.
  \end{equation}
  From \eqref{tuebingen2}, \eqref{eq:bndricc-r1}, and
  \eqref{eq:bndricc-r2} we now obtain, for $|s-r|<\delta$,
  \[\|P_r-P_s\|\leq\frac{1}{2\pi}\biggl(2t_0\eps+|s-r|\int_{|t|\geq t_0}
    \frac{4MM'}{|t|^{2-p}}\,dt\biggr).\]
  Consequently, the mapping $r\mapsto P_r$ is continuous.
  Since $G(X_r)=\range(P_r)$ are invariant under $T_r$ and its resolvent,
  Lemmas~\ref{lem:bndsol1},~\ref{lem:bndsol2} apply;
  using $c_{rC}\leq c_C$ for $r\in[0,1]$,
  we obtain
  a constant $L=L(A,p,c_C)>0$ independent of $r\in[0,1]$ such that if
  $X_r$ is bounded, then $\|X_r\|\leq L$.
  Hence Proposition~\ref{prop:contgraph} yields that all $X_r$ are bounded.

  To show the uniqueness of $X_+$, suppose that $X$ is another bounded nonnegative 
  solution of \eqref{eq:bndricc} with $X\mdef(A)\subset\mdef(A^*)$.
  Let $\varphi:H \to G(X)$, $\varphi u=(u,Xu)$ be the isomorphism defined in \eqref{last} in the 
  proof of Lemma~\ref{lem:bndsol1}. Our assumptions on $X$ imply that
  $G(X)$ is $T$-invariant and that $\varphi^{-1}T\varphi=A+BX$.
  Consequently, $\sigma(T|_{G(X)})=\sigma(A+BX)$.
  The operator $A^*X$ is closable and $\mdef(A) \subset \mdef(A^*X)$, hence $A^*X$ is 
  $A$-bounded. Thus, since $0\in\varrho(A)$, there exists $c>0$ such that
  \(\|A^*Xu\|\leq c\|Au\|\) for $u\in\mdef(A)$.
  Together with the $p$-subordination property of $B$ to~$A^*$, we then find
  \[\|BXu\|\leq c_B\|Xu\|^{1-p}\|A^*Xu\|^p
  \leq c_Bc^p\|X\|^{1-p}\|u\|^{1-p}\|Au\|^p,\qquad u\in\mdef(A),\]
  i.e.\ $BX$ is $p$-subordinate to $A$.
  Since $A$ is sectorial with angle $\theta<\pi/2$, 
  Remark~\ref{rem:psub} implies that
  $A+BX$ is sectorial with some angle $\theta'\in[\theta,\pi/2[\,$ and 
  radius $r>0$, in particular,
  \begin{equation}\label{eq:uniqu-sectincl}
    \sigma(A+BX)\subset\Sigma_{\theta'}\cup\overline{B_r(0)}.
  \end{equation}
  We will now  show that $\sigma(A+BX)$ is, in fact, contained in the open right half-plane;
  this shows, in particular, that $A+BX$ is sectorial with radius $0$.
  In view of \eqref{eq:uniqu-sectincl}, it is
  sufficient to prove that $\lambda\in r(A+BX)$ whenever $\Real \lambda\leq0$.
  Suppose, to the contrary, that $\lambda\in\C$ with $\Real \lambda \le 0$ and there exist $u_n\in\mdef(A)$, $\|u_n\|=1$, with
  \begin{equation}\label{eq:uniqu0}
    \lim_{n\to\infty}(A+BX-\lambda)u_n= 0.
  \vspace{-1mm}  
  \end{equation}
  This implies
  \begin{equation}\label{eq:uniqu1}
    \lim_{n\to\infty}((A+BX-\lambda)u_n|Xu_n)=0.
  \end{equation}
  Moreover, the Riccati equation \eqref{eq:bndricc} for $X$ yields that
  \begin{equation}\label{eq:uniqu2}\begin{aligned}
    \Real\big((A+BX)u|Xu\big) &= \frac{1}{2}\bigl((Au|Xu)+(Xu|Au)\bigr)+(BXu|Xu)\\
    &=\frac{1}{2}\bigl((Cu|u)+(BXu|Xu)\bigr),\quad u\in\mdef(A).
  \end{aligned}\end{equation}
  Combining \eqref{eq:uniqu1} and \eqref{eq:uniqu2} and using $B$, $C\ge 0$, we find
  \begin{align*}
    \limsup_{n\to\infty}\,\bigl((\Real\lambda)\!\cdot\!(Xu_n|u_n)\bigr)
    &=\limsup_{n\to\infty}\,\Real\big((A+BX)u_n|Xu_n\big)\\
    &=\limsup_{n\to\infty}\,\frac{1}{2}\bigl((Cu_n|u_n)+(BXu_n|Xu_n)\bigr)\\
    &\geq\frac{1}{2}\limsup_{n\to\infty}\,(BXu_n|Xu_n)\geq0.
  \end{align*}
  On the other hand,  $X\geq0$ and $\Real\lambda\leq0$ imply 
  $\limsup_{n\to\infty}\!\big( (\Real\lambda)\!\cdot\!(Xu_n|u_n)\big)\leq0$;
  consequently, $\lim_{n\to\infty}\|B^{1/2}Xu_n\|=0$.
  This, the fact that $B^{1/2}(A^*-\overline{\lambda})^{-1}$ and thus $(A-\lambda)^{-1}B^{1/2}$ is bounded, and \eqref{eq:uniqu0} now yield that
  \[
  \lim_{n\to\infty} u_n \!=\! \lim_{n\to\infty}\!\big(u_n+(A-\lambda)^{-1}B^{1/2}\!\cdot\!B^{1/2}Xu_n \big) \!=\!
  (A-\lambda)^{-1} \! \lim_{n\to\infty} (A+BX-\lambda) u_n = 0, 
  \vspace{-2mm}
  \]
  a contradiction to $\|u_n\|=1$.

  Since $A+BX$ is sectorial of angle $<\pi/2$ and its spectrum is contained
  in the open right half-plane,
  $(A+BX-\lambda)^{-1}$ and hence $(T|_{G(X)}-\lambda)^{-1}$ 
  are uniformly bounded on $\C_-$.
  Therefore $(T-\lambda)^{-1}(u,Xu)$ extends to a bounded analytic function
  on $\C_-$ which, in turn, implies that $(u,Xu)$ belongs to the
  spectral subspace $V_+=G(X_+)$ of $T$, see 
  \cite[Sect.~2 and Theorem~3.1]{bart-gohberg-kaashoek}.
  This proves that $G(X)\subset G(X_+)$ and hence $X=X_+$ since the bounded operators $X$ and $X_+$ are both defined on the dense subset $\mdef(A)\subset H$. 

  The case when $-A$ is sectorial
  is a consequence of the first case since $-X_-$ is nonnegative and 
  a solution of the Riccati equation corresponding to 
  \[\pmat{-A&B\\C&A^*}. \vspace{-8mm} \]
\end{proof}

\begin{remark}
In \cite[Section~7]{MR1809118} the existence and uniqueness of solutions of Riccati equations was proved under the weaker assumption that
$B$, $C$ are uniformly accretive,  but only in the case where all entries $A$, $B$, and $C$ are bounded and using a different approach. 
\end{remark}



\section{Examples}
\label{sec:example}

To illustrate the results of the previous sections, we consider 
three examples of Hamiltonians involving partial differential
and multiplication operators.
None of these examples is covered by the earlier results in
\cite{bubak-mee-ran,kuiper-zwart,langer-ran-rotten,wyss-rinvsubham,wyss-unbctrlham}.

In all examples, $B$ and $C$ are unbounded and hence
\cite{bubak-mee-ran,kuiper-zwart,langer-ran-rotten} cannot be applied.
Moreover, $B$ and $C$ do not map into an extrapolation space of $H$ 
and thus do not fit
into the setting of \cite{wyss-unbctrlham}.
In the first example,
the operator $A$ has continuous spectrum and hence no Riesz basis of
generalised eigenvectors exists as required in \cite{wyss-rinvsubham}.

\begin{example}
We consider the Hilbert space $H=L^2(\R^n)$ and the operators
\begin{alignat*}{2}
  Au&=(-\Delta+\eps)u,\qquad &&\mdef(A)=W^{2,2}(\R^n),\\
  Bu&=g_1u, &&\mdef(B)=\set{u\in H}{g_1u\in H},\\
  Cu&=g_2u, &&\mdef(C)=\set{u\in H}{g_2u\in H},
\end{alignat*}
with $\eps>0$ and nonnegative locally integrable functions 
$g_1,g_2:\R^n\to\R$.
Suppose, in addition, that $g_1$ is positive almost everywhere
and that $g_1$, $g_2$ satisfy estimates
\[
  \int_{B_r(x_0)}|g_j(x)|^2\,dx\leq c\,r^s,\quad
                                           x_0\in\R^n,\ 0<r<1,\ j=1,2,
\]
with  constants $c>0$ and $s\in[0,n]$ such that $s>n-4$;
e.g.\ one could choose $g(x)=|x|^{-q}$ with $0<q<\min\{2,n/2\}$ and $s=n-2q$.

The operator $A$ is positive and selfadjoint, $0\in\varrho(A)$, and
the a priori estimate
\[\|u\|_{W^{2,2}(\R^n)}\leq c_0\|Au\|,\qquad u\in\mdef(A),\]
can easily be verified by Fourier transformation. The multiplication operators 
$B$ and $C$ are selfadjoint, $B$ is positive and $C$ nonnegative.
By Lemma~\ref{lem:multop}, $B$ and $C$ are $p$-subordinate to $A$
with $p=\frac{1}{4}(n-s)<1$.
We can thus apply  our results
to the Hamiltonian
\[T=\pmat{A&B\\C&-A};\]
note that condition \eqref{eq:approxcontr} holds because $\ker B=\{0\}$.

Hence Theorem~\ref{theo:bndricc} yields the existence
and uniqueness of a bounded nonnegative selfadjoint 
solution $X_+$ of the Riccati equation
\[(AX_++X_+A+X_+BX_+-C)u=0,\qquad u\in\mdef(A);\]
Theorem~\ref{theo:ham-ricc} yields the existence of a nonpositive selfadjoint solution $X_-$ of
the Riccati equation
\[
 (AX_-+X_-(A+BX_-)-C)u=0,\qquad u\in\mdef(A)\cap X_-^{-1}\mdef(A).
\]
%
\end{example}

\smallskip

In the next two examples, $A$, and hence also $T$,
has compact resolvent and pure point spectrum.
However, $A$ is not normal as required in the known existence results
for Riesz bases of generalised eigenvectors, e.g.\ 
\cite[Theorem~6.12]{markus},
\cite[Theorem~6.1]{wyss-psubpert},
and thus \cite{wyss-rinvsubham} 
cannot be applied.

\begin{example}
  Let $\Omega\subset\R^n$ be open and bounded with smooth boundary
  $\partial \Omega$ such that no point of $\partial \Omega$ belongs to
  the interior of $\overline{\Omega}$.
  Let $H=L^2(\Omega)$ and consider the \vspace{1mm} operators
  \begin{alignat*}{1}
    Au&=\Delta^2u, \qquad \hspace{1.3mm}
    \mdef(A)=\set{u\in W^{4,2}(\Omega)}{u=\Delta u+f\partial_\nu u=0
      \text{ on }\partial\Omega},\\[1mm]
    Bu&=-\sum_{j,k=1}^n\partial_j(g_{jk}\partial_ku) + g_0u,\qquad \hspace{0.3mm}
    \mdef(B)=W^{2,2}(\Omega)\cap W_0^{1,2}(\Omega),\\[-1mm]
    Cu&=-\sum_{j,k=1}^n\partial_j(h_{jk}\partial_ku) + h_0u,\qquad 
    \mdef(C)=W^{2,2}(\Omega)\cap W_0^{1,2}(\Omega),
  \end{alignat*}
  where $f\in C^\infty(\partial\Omega)$ with $\Real f\geq0$, 
  $g_{jk},h_{jk},g_0,h_0\in C^\infty(\Omega)$,
  $g_0,h_0\geq0$, and the matrices $(g_{jk})_{j,k=1\dots n}$ and
  $(h_{jk})_{j,k=1\dots n}$ are positive definite and positive semidefinite,
  respectively, almost everywhere on $\Omega$.
  The outward normal derivative is~$\partial_\nu$.

  From the theory of elliptic partial differential operators,
  see e.g. \cite{lions-magenes}, it follows that
  $B$ and $C$ are selfadjoint, $B$ is positive, and $C$ is nonnegative.
  The operator $A$ is closed and its adjoint is given by
  \[A^*u=\Delta^2u, \qquad 
    \mdef(A^*)=\set{u\in W^{4,2}(\Omega)}{u=\Delta u+\bar{f}\partial_\nu u=0
      \text{ on }\partial\Omega};\]
  note that $A$ is not selfadjoint (not even normal) if $\Imag f\neq0$.
  Integration by parts shows that
  \begin{equation}\label{eq:sectex-form}
  (Au|u)=\int_\Omega|\Delta u|^2\,dx
  +\int_{\partial \Omega}f|\partial_\nu u|^2
  \,d\sigma,\qquad u\in\mdef(A).
  \end{equation}
  Consequently, there exist $c_0,c_1>0$ such that for $u\in\mdef(A)$
  \begin{equation}\label{eq:sectex-est}
    \Real(Au|u)\geq\|\Delta u\|^2\geq c_0\|u\|_{W^{2,2}(\Omega)}^2,
    \qquad |\Imag(Au|u)|\leq c_1\|u\|_{W^{2,2}(\Omega)}^2.
  \end{equation}
  This implies that $\ker A=\{0\}$, that $\range(A)$ is closed, and that 
  the numerical range $W(A)$ is contained in a sector $\Sigma_\theta$, 
  see \eqref{sector}, more precisely,
  \begin{equation}
  \label{ex:final}
    W(A)\subset\set{z\in\Sigma_\theta}{\Real z\geq c_0} \quad \text{with} \ \theta=\arctan(c_1/c_0).
  \end{equation}  
  Since \eqref{eq:sectex-form} (with $f$ replaced by $\bar f$)
  and \eqref{eq:sectex-est} also hold for $A^*$,
  this yields
  $\range(A)^\perp=\ker A^*=\{0\}$ and thus $0\in\varrho(A)$.
  In view of Remark~\ref{rem:sect}~(iii) 
  we obtain that $A$ is sectorial with angle~$\theta$.
  Finally, \eqref{eq:sectex-est} also implies
  \[\|u\|_{W^{2,2}(\Omega)}\leq c_0^{-1}\|u\|^{1/2}\|Au\|^{1/2},\qquad
  u\in\mdef(A),\]
  and the same with $A$ replaced by $A^*$. Consequently, $B$ and $C$ are
  $\frac 12$\,-\,sub\-ordinate to $A^*$ and $A$, respectively. Moreover, $\ker B = \{0\}$ since 
  $B$ is positive and thus assumption \eqref{eq:approxcontr} is satisfied.
  
  Hence Theorem~\ref{theo:bndricc} shows that there exists a
  unique bounded nonnegative selfadjoint
  solution $X_+$ of the Riccati equation
  \[(A^*X_++X_+A+X_+BX_+-C)u=0,\qquad u\in\mdef(A);\]
  Theorem~\ref{theo:ham-ricc} shows that there exists a nonpositive selfadjoint solution $X_-$
  of the Riccati equation
  \[(A^*X_-+X_-(A+BX_-)-C)u=0,\qquad u\in\mdef(A)\cap X_-^{-1}\mdef(A^*).\]

%
\end{example}

\smallskip

In our final example, we consider a Riccati equation with coefficients $\wt A$, $\wt B$, and $\wt C$ such that $\wt A$
is sectorially dichotomous, but neither $\wt A$ nor $-\wt A$
are sectorial. Hence Theorem~\ref{theo:bndricc} does not apply and both
solutions $X_\pm$ will be unbounded in general.

\begin{example}
  Let $\Omega \subset \R^n$ and the operators $A$, $B$, $C$ in $L^2(\Omega)$
  be as in the previous example.
  Consider the block operator matrices
  \[\wt A=\pmat{A&0\\0&-A^*},\quad
  \wt B=\pmat{B&\beta B\\\beta B&B},\quad
  \wt C=\pmat{C&\gamma C\\\gamma C&C} \quad
  \text{in }L^2(\Omega)\times L^2(\Omega)\]
  where $\beta\in[0,1[$, $\gamma\in[0,1]$.
  Then $\wt A$ is sectorially dichotomous, $\wt B$, $\wt C$ are symmetric,
  $\wt B$ is positive since $\beta<1$, $\wt C$ is nonnegative, and 
  $\wt B$, $\wt C$ are $\frac{1}{2}$-subordinate to $\wt A^*$, $\wt A$, respectively.
  By Theorem~\ref{theo:ham-ricc} the Riccati equation
  \begin{align*}
    &\bigl( \wt A^* X_\pm+X_\pm\bigl( \wt A+\wt BX_\pm\bigr)-\wt C\bigr)u=0,
     \qquad u\in\mdef\bigl(\wt A\bigr)\cap X_\pm^{-1}\mdef\bigl(\wt A^*\bigr),
  \end{align*}
  has two selfadjoint solutions $X_\pm$, where $X_+$ is nonnegative and
  $X_-$ is nonpositive, both of which are  unbounded in general.
\end{example}



\smallskip

\noindent
{\bf Acknowledgement.} \
The support for this work by Deutsche Forschungsgemeinschaft (DFG), Grant TR 368/6-2, and Schweizerischer Nationalfonds (SNF), Grant 200021-119826/1, is gratefully acknowledged.
We would also like to thank the referee very much for useful comments.

{\small 
\bibliographystyle{cwyss}
\bibliography{dichotham}
}

\end{document}